\documentclass{article}
\usepackage{graphicx,float}
\usepackage{lineno,hyperref}
\usepackage{amsthm}
\usepackage{afterpage}
\usepackage{amssymb}
\usepackage{amsmath, amsthm, amscd, amsfonts, xpatch}
\usepackage{bussproofs}
\usepackage{bm}
\usepackage[english]{babel}
\usepackage{csquotes}
\usepackage{caption}
\usepackage{dirtytalk}
\usepackage{ dsfont }
\usepackage{enumerate}
\usepackage{enumitem}
\usepackage{enumitem,kantlipsum}
\usepackage{ifxetex}
 \usepackage{latexsym}
\usepackage[capitalise]{cleveref}
\usepackage{mathtools}
\usepackage{mathrsfs}
\usepackage{mfirstuc}
\usepackage{pst-node}
\usepackage{pb-diagram}
\usepackage{pgfplots}
\usepackage{pst-node}
\usepackage{setspace}
\usepackage{stmaryrd}
\usepackage{tikz}
\usepackage{tkz-euclide}
\usepackage{tipa}
\usepackage{tikz-cd}
\usepackage{titlecaps}
\usepackage{upgreek}
\usepackage{array}
\usepackage[nottoc]{tocbibind}
\patchcmd{\thebibliography}{\section{Bibliography}}{}{}{}

\usepackage{mystyle}
\DeclareMathAlphabet{\mathpzc}{OT1}{pzc}{m}{it}
\DeclareMathOperator{\vol}{vol}
\DeclareMathOperator{\fr}{Frac}
\DeclareMathOperator{\IMM}{Im}
\DeclareMathOperator{\sg}{Sgn}
\DeclareMathOperator{\area}{area}
\DeclareMathOperator{\maxx}{Max}
\DeclareMathOperator{\blu}{Blue\ part}
\DeclareMathOperator{\tang}{Tang}
\input xy
\usetikzlibrary{calc,patterns,angles,quotes}

\xyoption{all}
\newtheorem*{theorem*}{Theorem}
\newtheorem{theorem}{Theorem}[section]
\newtheorem{lemma}{Lemma}[section]

\newtheorem{proposition}{Proposition}[section]
\newtheorem*{proposition*}{Proposition}

\newtheorem{observation}{Observation}[section]
\newtheorem{fact}{Fact}[section]

\newtheorem{definition}{Definition}[section]  \newtheorem{notation}{Notation}[section]
\newtheorem{example}{Example}[section]

\newtheorem{remark}{Remark}[section]

\numberwithin{equation}{section}

\theoremstyle{remark}

\newcommand{\Addresses}{{
  \bigskip
  \footnotesize

  M. Mehrabdollahei, \textsc{Sorbonne Université, IMJ-PRG,  Paris cédex 05, France}\par\nopagebreak
  \textit{E-mail address},  M. Mehrabdollahei: \texttt{mahya.mehrabdollahei@imj-prg.fr}
\maketitle
}}

 \pgfplotsset{compat=1.17} 
\definecolor{royalfuchsia}{rgb}{0.79, 0.17, 0.57}
\definecolor{lemonchiffon}{rgb}{1.0, 0.98, 0.8}
\definecolor{violet(ryb)}{rgb}{0.53, 0.0, 0.69}
	\definecolor{teagreen}{rgb}{0.82, 0.94, 0.75}
	\definecolor{toolbox}{rgb}{0.45, 0.42, 0.75}
	\definecolor{yellow(ncs)}{rgb}{1.0, 0.83, 0.0}
	\definecolor{wisteria}{rgb}{0.79, 0.63, 0.86}
	\definecolor{lightmauve}{rgb}{0.86, 0.82, 1.0}
	\definecolor{mustard}{rgb}{1.0, 0.86, 0.35}
		\definecolor{richbrilliantlavender}{rgb}{0.95, 0.65, 1.0}
		\definecolor{upforestgreen}{rgb}{0.0, 0.27, 0.13}
		\definecolor{steelblue}{rgb}{0.27, 0.51, 0.71}
		\definecolor{mulberry}{rgb}{0.77, 0.29, 0.55}
		\definecolor{satinsheengold}{rgb}{0.8, 0.63, 0.21}	\definecolor{byzantium}{rgb}{0.44, 0.16, 0.39}
		\definecolor{sapphire}{rgb}{0.03, 0.15, 0.4}

  \newcommand\blfootnote[1]{%
  \begingroup
  \renewcommand\thefootnote{}\footnote{#1}%
  \addtocounter{footnote}{-1}%
  \endgroup
}

\usetikzlibrary{shapes.misc}

\tikzset{cross/.style={cross out, draw=black, minimum size=2*(#1-\pgflinewidth), inner sep=0pt, outer sep=0pt},
cross/.default={1pt}}
\title{Mahler measure of $P_d$ polynomials}

\author{
 Mahya Mehrabdollahei}

\begin{document}
\maketitle
  
\begin{abstract}
This article investigates the Mahler measure of a family of 2-variate polynomials, denoted by $P_d$, for $d\geq 1$, unbounded in both degree and genus. By using a closed formula for the Mahler measure \cite{1}, we are able to compute $m(P_d)$, for arbitrary $d$, as a sum of the values of dilogarithm at special roots of unity. We prove that $m(P_d)$ converges,  and the limit is proportional to $\zeta(3)$, where $\zeta$ is the Riemann zeta function.
\end{abstract}
\section{Introduction}
Mahler measure is an interesting notion,  used in number theory, analysis, special functions, random walks, etc. The book \cite{0} may be seen as a reference for prerequisite materials and recent achievements in Mahler measure theory. The (logarithmic) Mahler measure  of a multi-variate polynomial, $P(x_1,\ldots,x_n) \in \mathbb{C}[x_1,\ldots,x_n]$, denoted by $m(P)$, is defined by the following formula :
$$m(P) =  \frac{1}{(2\pi)^n} \left(\int_0^{2\pi} \int_0^{2\pi} \cdots \int_0^{2\pi} \log  \bigl |P(e^{i\theta_1}, e^{i\theta_2}, \ldots, e^{i\theta_n}) \bigr|  \ \  d\theta_1\, d\theta_2\cdots d\theta_n \right).$$
It is possible to prove that this integral is not singular, and $m(P)$ always exists \cite{11}, but there is no general closed formula to compute it. Moreover, it is not easy to approximate it with arbitrary precision.
Guilloux and Marché \cite{1} found a closed formula for a specific class of 2-variate polynomials, called regular exact polynomials, which expresses their Mahler measure as a finite sum. Boyd  and Rodriguez-Villegas \cite{boyd_rodriguez-villegas_2002} found another closed formula for the Mahler measure of a specific family of exact 2 variable polynomials with a different language.  Boyd \cite{7}, Bertin, and Zudilin \cite{3,4} investigated families of curves of genus 2. Furthermore, Bertin \cite{5} computed the Mahler measure of a family of 3-variate polynomials $Q_k(x, y, z)$, defining $K3$ surfaces. Lalín \cite{8} developed a new method for expressing Mahler measures of some families of polynomials in terms of polylogarithms. Also \cite{boyd_rodriguez-villegas_2002} and \cite{Mahler2} give many information about the relation between Mahler measures of exact polynomials and the values of Dilogarithm function at certain algebraic numbers.  \\
In this article we study a specific
 family of 2-variate exact polynomials. We compute their Mahler measure. Furthermore, we compute the limit of the Mahler measure of this family. 
 In \cref{sec:exact}, we introduce the family  $P_d(x,y):=\sum_{0 \leq i+j\leq d} x^i y^j$, presented to us by François Brunault. He noted that $P_d$ is exact. To apply the formula in \cite{1}, we need to determine its terms. To do so, we compute the toric points, a volume function, and a kind of sign function. \cref{Toric points} is devoted to the computation of $m(P_d)$. In \cref{6}, to achieve the objective, finding a new explicit formula to compute $m(P_d)$ in terms of the values of the Dilogarithm at roots of unity,  we introduce  $\vol(\theta, \alpha):=D(e^{i\theta})-D(e^{i(\theta+\alpha)})+D(e^{i\alpha})$. Using this function we prove the following theorem in \cref{7};
\begin{theorem*} The Mahler measure $m(P_d)$ is expressed in terms of $\vol$ as follows: 
 $$ 2\pi m(P_d)= \frac{-2}{d+2} \sum_{0<k<k'\leq d} \vol\left(\frac{2k\pi}{d+1},\frac{2(k'-k)\pi}{d+1}\right)+\frac{2}{d+1} \sum_{0<k<k'\leq d+1} \vol\left(\frac{2k\pi}{d+2},\frac{2(k'-k)\pi}{d+2}\right).$$

 \end{theorem*}
 The above theorem asserts that the Mahler measure of $P_d$ for arbitrary $d\geq 1$ can be expressed in terms of the finite sum over the Dilogarithm function at certain roots of unity. This theorem is the key point for connecting the values of $m(P_d)$ to special values of $L$-functions, which is a reminiscence of the work of Smyth \cite{smyth_1981} and Boyd \cite{7,9}. \\
  Moreover, in the above formula, when $d$ goes to infinity, each summation is proportional to a Riemann sum of $\vol$. Hence, $\lim_{d\rightarrow \infty} m(P_d)$ is written as a subtraction of two expressions; each of them is proportional to a Riemann sum of $\vol$ and both go to infinity. In order to find the limit, we first use the Rieman sum technics and later by analysing the errors, we prove the following theorem:
\begin{theorem*} 
The $\lim_{d \rightarrow \infty} m(P_d)$ exists and we have:
\normalfont
\begin{align}\label{limizeta}
    \lim_{d\rightarrow \infty}m(P_d)=\frac{9}{2\pi^2} \zeta (3). 
\end{align}
\end{theorem*}

A theorem of Boyd and Lawton expresses limits of Mahler measure of univariate polynomials as the Mahler measure of a multivariate polynomial;
\begin{theorem*}\cite{9,10}
For $P \in  \mathbb{C}[x_1, . . . , x_n]$, we have:
$$\lim _{k_2\rightarrow \infty}\cdots \lim _{k_n\rightarrow \infty}m(P(x, x^{k_2}
, . . . , x^{k_n} )) = m(P(x_1, . . . , x_n)).$$
\end{theorem*}
Our result is similar in spirit. After having determined the value of the limit of $m(P_d)$, we noted that it is indeed the value of a Mahler measure. D’Andrea and Lalín \cite{6} defined a polynomial in $4$ variables;  $P_\infty:=(x-1)(y-1)-(z-1)(w-1)$. They proved that $m(P_\infty )=\frac{9}{2\pi^2}\zeta(3)$. 
The fact that $\lim_{d \rightarrow \infty }m(P_d)=m(P_\infty)$ is in the spirit of Boyd-Lawton theorem, explained below. After noticing this coincidence, one may link $P_d$'s and $P_{\infty}$ as follows:
\begin{align*}
 P_d(x,y)=\frac{P_\infty(x^{d+2},y,x,y^{d+2})}{(1-x)(1-y)(x-y)},
\end{align*}
and since the Mahler measures of the denominator is zero we have:
\begin{align}\label{**}
    m(P_d(x,y))=m(P_\infty(x^{d+2},y,x,y^{d+2})).
\end{align}
 However, let us stress that for applying the Boyd-Lawton theorem to $P_d$ we face two types of difficulties; First, it would be hard to guess a candidate for $P_{\infty}$ without doing the computation of $\lim_{d \rightarrow \infty}m(P_d)$; Second, finding $P_{\infty}$ is not sufficient to claim that $\lim_{d \rightarrow \infty}m(P_d)=m(P_{\infty})$. We can not apply the theorem directly as it is written in the literature, since $P_d$ is a family of two variate polynomials. In fact, we need to have a generalization of Boyd-Lawton theorem. We will discus this generalization in a forthcoming article in collaboration with Guilloux, Brunault, Pengo.\\
 In contrast to $P_d$ family,  Gu and Lalìn \cite{Gu} studied a $3$-variable family of polynomials with two parameters, $x^{a+b} +1+(x^a +1)y+(x^b -1)z$. They applied the actual form of the theorem of Boyd-Lawton on this family, and proved that $\lim \limits_{\substack{a\rightarrow \infty \\ b\rightarrow \infty}}x^{a+b} +1+(x^a +1)y+(x^b-1)z=m(P_{\infty})$.
 
\begin{acknowledgments}
I would like to express my deep gratitude to Antonin Guilloux and Fabrice Rouillier, my research supervisors, for their patient guidance, enthusiastic encouragement and useful critiques of this research work.\\
I am very grateful to Francois Brunault and Riccardo Pengo for their several helpful discussions, in particular about $P_\infty$, which was the starting point of our future common article.\\
I am thankful to Marie-Jose Bertin for the extraordinary experiences she arranged for me and for her interesting comments and questions which provide me the opportunity to start our common research.\\
This project has received funding from the European Union's Horizon 2020 research and innovation program under the Marie Skłodowska-Curie grant agreement No 754362.\blfootnote{
  \includegraphics[height=4.0mm]{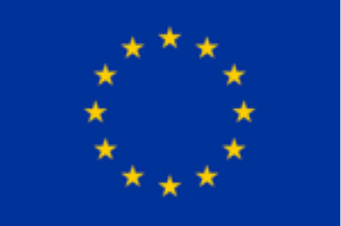}}

\end{acknowledgments}

\section{Exact polynomials }
\label{sec:exact}
In this section, the class of exact polynomials is introduced.
\begin{definition}
The real differential $1$-form $ \eta$ on ${{\mathds{C}}^*}^2 $ is defined by $ \eta=\log|y|\ d\arg(x)-\log |x|\ d\arg(y)$.
\end{definition}
\begin{remark}
 Let $P \in \mathds{C} [x,y]$ and $C$ be the algebraic curve defined by
 $$C= \{(x,y)\in {{\mathds{C}}^*}^2| P(x,y)=0 , dP(x,y) \neq 0 \};$$
 Then the form $\eta$ restricted to $C$ is closed.
 \end{remark} 
After the previous remark one may ask about the exactness of $\eta|_C $. In general, the answer is no, but this question leads to the definition of exact polynomials.
\begin{definition}
A polynomial $P\in \mathds{C} [X,Y]$ is called \textbf{exact} if the form $\eta$ restricted to the algebraic curve $C$ is exact. In this case, any primitive for $\eta$ is called a \textbf{Volume function} associated with the exact polynomial $P$. 
\end{definition}
To see a simple example of exact polynomials, we need the following definition.
\begin{definition}\label{def 2.3}
 The \textbf{Bloch-Wigner} Dilogarithm function $D(z)$ is defined by: 
\begin{equation*}
    D(z) = \IMM(Li_2(z)) + \arg(1-z)\log |z|, 
\end{equation*}
where $\arg$ denotes the branch of the argument, lying between $-\pi$ and $\pi$, and $ Li_2(z)$ is the following function:
\begin{equation*}
    Li_2(z) = - \int_{0}^{z} \log(1-u)\frac{du}{u} \ \ \  \ for \ z \in \mathds{C} \setminus [1,\infty).
\end{equation*}
\end{definition}
The function $D(z)$ is real analytic on $\mathds{C}$ except at the two points $0$ and $1$, where it is continuous but not differentiable.
We briefly summarize some needed properties of this function. For more information see \cite{2}. 

\begin{enumerate}
    \item $D(\bar{z})=-D(z)$.\\ 
    \item If $|z|=1$,
$D(z)=D(e^{i\theta})=\sum_{n=1}^{\infty} \frac{\sin (n\theta)}{n^2}$, 
in particular $D(e^{k\pi i})=0$.
\end{enumerate}
The link between the differential of $D$ and $\eta$ is well known, see \cite{12} or \emph{Theorem} $7.2$ of \cite{0};
\begin{fact} $-D(z)$ is a primitive for $\eta$ restricted to \\ $\{(z,1-z)\in \mathds{C^*}^2 \}$, i.e.,
\begin{equation*}
    -dD(z)= \eta(z,1-z).
    \end{equation*}
    \end{fact}

\begin{example}
\label{ex 2.1}The polynomial $P_1(x,y)=x+y+1$, is exact and a volume function is $-D(-x)$; (See \cref{th 2.2}).
\end{example}

\subsection{\texorpdfstring{$P_d$}{} polynomials}

We generalize the first example $P_1$ to a family of polynomials, called $P_d(x,y)$, with $d\geq 1$;
\begin{notation}
$$P_d(x,y):=\sum_{0 \leq i+j\leq d} x^i y^j.$$
\end{notation}

Let us prove the exactness of $P_d$, for all $d$.
\subsubsection{Exactness of \texorpdfstring{$P_d$}{}}
The best way to prove the exactness of $\eta$ is by an abstract algebraization of $\eta$.
Consider the multiplicative group ${K_d}^*$ of the field $K_d=\fr \frac{\bar{\mathds{Q}}[X,Y]}{<P_d>}$, as a $\mathds{Z}$-module.
The second exterior product of ${K_d}^*$ is ${K_d}^* \wedge {K_d}^*$. Note that the associated group operation in ${K_d}^*$ and ${K_d}^* \wedge {K_d}^*$ are respectively multiplication and addition.
 Consider the alternative bi-linear map $\imath:{K_d}^* \times {K_d}^* \rightarrow {\Upomega}_C^1 $ defined by:
\begin{align*}
   \imath \colon & {K_d}^* \times {K_d}^* \to {\Upomega}_C^1 \\
  &(f,g) \mapsto \log|g| d\arg f- \log |f| d\arg g.
\end{align*}
Where, $d\arg f=\IMM(d\log f)=\IMM(df/f)$, $C$ is the curve of $P_d$, minus the set of zeros and poles of $f$ and $g$.
Moreover, ${\Upomega}_C^1 $ is the $\mathds{C}$-vector space of smooth differential one-forms on $C$. According to the universal property of the exterior product, there is a unique morphism of $\mathds{Z}$-modules, $\bar{\imath}:{K_d}^* \wedge {K_d}^* \rightarrow {\Upomega}_C^1 $, such that the following diagram commutes. 

\[
  \begin{tikzcd}
     {K_d}^* \wedge {K_d}^* \arrow{dr}{\bar{\imath}} & \\
    {K_d}^* \times {K_d}^*  \arrow{r}[swap]{\imath} 
    \arrow{u}{\bigwedge}& {\Upomega}_C^1
  \end{tikzcd}
\]

where $\bigwedge$ is defined by:
\begin{align*}
   \bigwedge \colon & {K_d}^* \times {K_d}^* \to {K_d}^* \wedge {K_d}^* \\
  &(f,g) \mapsto f \wedge g
\end{align*}
 Note that according to the definitions of $\imath (f,g)$ and $\eta$ we have $ \eta_{(f,g)}= \imath (f,g)$.\\
 
By using the the universal property of wedge product, in the next lemma we check that torsion elements of  ${K_d}^* \wedge {K_d}^*$ belong to the kernel of $\bar{\imath}$. 

\begin{lemma} \label{lem2.1}
If  $f\wedge g = f' \wedge g'$, then $\imath(f,g)=\imath(f',g')$. Moreover, if $\sum_{i=1}^{n}  \epsilon_i f_i\wedge g_i =0$, then $\sum_{i=1}^{n}\epsilon_i \imath(f_i,g_i) =0$, where $\epsilon_i \in \mathds{Z}$, especially the torsion elements of ${K_d}^* \wedge {K_d}^*$ are sent to zero by $\bar{\imath}$.
\end{lemma}
\begin{proof}
The first part is clear by the universal property. For the second part, $\bar{\imath}$ is a morphism of $\mathds{Z}$- module, so $\sum_{i=1}^{n}\epsilon_i\imath(f_i,g_i)=\sum_{i=1}^{n}\epsilon_i\bar{\imath}(f_i \wedge g_i)=\bar{\imath}(\sum_{i=1}^{n} \epsilon_i f_i \wedge g_i )=0$. Finally, if $\sum_{i=1}^{n}  \epsilon_i f_i\wedge g_i $  is a torsion element in ${K_d}^* \wedge {K_d}^*$, 
there is an integer $n$ such that $n\left(\sum_{i=1}^{n}  \epsilon_i f_i\wedge g_i \right)=0$. Thus,
$\bar{\imath}\left(n(\sum_{i=1}^{n}  \epsilon_i f_i\wedge g_i )\right)= n\left(\bar{\imath}(\sum_{i=1}^{n}  \epsilon_i f_i\wedge g_i )\right)=0$. Hence, the differential form $\bar{\imath}\left(\sum_{i=1}^{n}  \epsilon_i f_i\wedge g_i \right)$ is a torsion element in the $\mathds{C}$-vector space ${\Upomega}_C^1$, so $\bar{\imath}\left(\sum_{i=1}^{n}  \epsilon_i f_i\wedge g_i \right)=0$.
\end{proof}
\begin{example}\label{ex}
  For all $g \in K_d^*$, $g \wedge -1$  is a torsion element in ${K_d}^* \wedge {K_d}^*$.
\end{example}
\begin{proof} We have
$g \wedge 1 = g \wedge (1 \cdot 1) = g \wedge 1 + g\wedge 1$, so $g \wedge 1=0$. Moreover we have
$0 =g\wedge 1 =(g \wedge (-1)(-1))= (g \wedge -1) + (g \wedge -1)= 2.(g\wedge -1)$,  so $g\wedge -1$ is a torsion element.
\end{proof}
The following theorem is the key to find a volume function and to generalize \cref{ex 2.1}.
\begin{proposition} \label{11111}
If $x,y \in {K_d}^*$ and $x \wedge y=\sum_{i=1}^{n} z_i\wedge (1-z_i )$ modulo some torsion elements in ${K_d}^* \wedge {K_d}^*$, then $(-\sum_{i=1}^{n}D(z_i))$ is a primitive form for $\eta$ restricted to smooth zeroes of $P_d(x,y)$.
\end{proposition}
\begin{proof}
We have:
$$\bar{\imath}(x \wedge y)=\bar{\imath}\left(\sum_{i=1}^{n}z_i\wedge (1-z_i )+\sum_{i=1}^{n} f_i \wedge g_i\right), \ \ \text{where $f_i \wedge g_i$ are torsion elements}. $$
Since $\bar{\imath}$ is a morphism of abelian groups, and $\sum_{i=1}^{n} f_i \wedge g_i$ is a torsion element, by \cref{lem2.1} and \cref{ex 2.1}, we have:
 \begin{align*}
     & \eta_{(x,y)}= \imath(x,y)  =\bar{\imath}(x \wedge y)= \bar{\imath}\left(\sum_{i=1}^{n}(z_i\wedge (1-z_i ))+\sum_{i=1}^{n} f_i \wedge g_i\right)\\
     & =\sum_{i=1}^{n}\bar{\imath}
     \left (z_i\wedge (1-z_i )\right)= \sum_{i=1}^{n}\imath (z_i,1-z_i ) = -(\sum_{i=1}^{n}d D(z_i)) =d(-\sum_{i=1}^{n} D(z_i)).
 \end{align*}
\end{proof}
\begin{remark}
We notice that $\wedge$ computation for finding a volume function does not depend on the torsion elements, so in the sequel of this section we use the notation $\doteq$ which refers to equality up to torsion elements; For example, for all $f,g$ we have $(-f) \wedge (-g)\doteq f\wedge (-g)\doteq(-f) \wedge g\doteq f \wedge g$.
\end{remark}
In the following lemma, we show two equalities needed for proving the exactness of $P_d$.
\begin{lemma}\label{lem 2.2 }
We have the following equalities: 
\begin{equation} \label{16}
     x \wedge (1- x/y)  -y \wedge (1- y/x ) \doteq x/y \wedge (1- x/y) - x \wedge y ,
 \end{equation}
 \begin{equation} \label{17}
 y \wedge (1-(y/x)^{d+1}) - x \wedge (1-(x/y)^{d+1}) \doteq (y/x) \wedge (1-(y/x)^{d+1}))+ (d+1) x \wedge y . 
\end{equation}
 \end{lemma}
 \begin{proof}
  We just prove the first equality, the second one is proved similarly. By replacing $x$ with $\frac{x}{y}.y$ we have:
\begin{align*}
   & x \wedge (1- x/y)  - y \wedge (1- y/x ) = x/y \wedge (1- x/y) + y \wedge (1- x/y) - y \wedge (1- y/x ) = \\
   & x/y \wedge (1- x/y) + y \wedge \frac{1- x/y}{1- y/x}
   = x/y \wedge (1- x/y) + y \wedge \frac{\frac{y-x}{y}}{\frac{x-y}{x}}=
    x/y \wedge (1- x/y) + y \wedge \frac{-x}{y}=\\
    & x/y \wedge (1- x/y) + y \wedge-1 + y \wedge x - y \wedge y\doteq
    x/y \wedge (1- x/y) - x \wedge y .
   \end{align*}
   
\end{proof}
We recover:
    \begin{example}\label{th 2.2}
The polynomial $P_1$ is exact and a volume function is $-D(-x)$.
\end{example}
In chapter 7 of \cite{0} and \cite{8}, one can find a similar method to prove this fact, and to compute the Mahler measure. 
\begin{proof}
We notice that in $K_1:=\fr \frac{\bar{\mathds{Q}}[X,Y]}{<P_1>}$, we have $1+x=-y$. It yields:
$$x \wedge y \doteq (-x) \wedge (-y) = (-x)\wedge  (1-(-x)).  $$
Then according to \cref{11111}, $-D(-x)$ is a volume function and $P_1$ is exact.
\end{proof}
The previous example generalizes to the whole family;
\begin{theorem}\label{th2.2}
For all  $d\geq 1$, $P_d$ is an exact polynomial, and for $d\geq 2$ a volume function, denoted by $V$, is defined as follows:
\begin{align*}
    V(x,y)=\frac{1}{(d+1)(d+2)}[D(y^{d+1})-D(x^{d+1})-D((y/x)^{d+1})]+\frac{1}{(d+2)}[D(x)-D(y)-D(x/y)].
\end{align*}
\end{theorem}
\begin{proof}
For $P_1$, we have already proved in \cref{th 2.2}, that is exact with the volume function $-D(-x)$. For $d\geq 2$ we have these equations:
\begin{equation*} 
    P_d(x,y)=P_{d-1}(x,y)+y^d(\frac{1-(x/y)^{d+1}}{1-(x/y)}) \ \ \text{and}\ \ 
    P_d(x,y)=y P_{d-1}(x,y)+(\frac{1-x^{d+1}}{1-x}).
\end{equation*}
Therefore, at smooth zeros of $P_d$, we have:
\begin{equation*} 
   0= P_d(x,y)=P_{d-1}(x,y)+y^d(\frac{1-(x/y)^{d+1}}{1-(x/y)})=y P_{d-1}(x,y)+(\frac{1-x^{d+1}}{1-x}).
\end{equation*}
In other words, we have:
\begin{equation} \label{11}
   P_{d-1}(x,y)=-y^d(\frac{1-(x/y)^{d+1}}{1-(x/y)})=
-1/y(\frac{1-x^{d+1}}{1-x}),
\end{equation}
hence,
\begin{equation} \label{13}
    y^{d+1}= \frac{1-x^{d+1}}{1-x} \frac{1-(x/y)}{1-(x/y)^{d+1}}.
\end{equation}
Instead of $x \wedge y$, we compute $\frac{1}{d+1} x \wedge y^{d+1}$.
By replacing $y^{d+1}$ from \cref{13} in  $\frac{1}{d+1} x \wedge y^{d+1}$ we have :
 \begin{equation} \label{14}
     x \wedge y =\frac{1}{d+1} \left(x \wedge (1-x^{d+1}) - x \wedge (1-x) +  x \wedge (1- x/y ) - x \wedge (1-(x/y)^{d+1})   \right).
 \end{equation}
 Because $P_d$, for $d\geq 1$, is a symmetric polynomial, so we can switch $x$ and $y$; Similarly, we have:
 \begin{equation}\label{15}
     y \wedge x =\frac{1}{d+1} \left(  y \wedge (1-y^{d+1})  - y \wedge (1- y) +  y \wedge (1- y/x ) - y \wedge (1-(y/x)^{d+1})   \right).
 \end{equation}
By subtracting \cref{15} from \cref{14}, we have:
 \begin{align}\label{d}
   2(d+1) (x \wedge y)= x \wedge (1-x^{d+1})-y \wedge (1-y^{d+1}) - x \wedge (1-x) + y \wedge (1-y) 
    \end{align}
  \begin{align*}
     +x \wedge (1- x/y) - y \wedge (1- y/x )  -x \wedge (1-(x/y)^{d+1}) + y \wedge (1-(y/x)^{d+1}). 
 \end{align*}
  By replacing  \cref{17} and \cref{16} in \cref{d} and simplifying (based on \cref{lem 2.2 }), we have :
\begin{align*}
    & (d+2)(x\wedge y)\doteq 1/(d+1)x^{d+1}\wedge (1-x^{d+1}) - 1/(d+1)y^{d+1} \wedge (1-y^{d+1})-x \wedge (1-x) +\\
    & y \wedge (1-y)+ x/y \wedge (1- x/y)+ 1/(d+1)(y/x)^{d+1} \wedge (1-(y/x)^{d+1})).
   \end{align*}
In other words, we have:
\begin{align*}
   & (x\wedge y)\doteq \frac{1}{(d+1)(d+2)} \left( x^{d+1} \wedge (1-x^{d+1})- y^{d+1} \wedge (1-y^{d+1})+(y/x)^{d+1} \wedge (1- (y/x)^{d+1})\right)\\
   &  +\frac{1}{d+2}\left( y \wedge (1-y)- x \wedge (1-x) + x/y \wedge (1- x/y ) \right).
\end{align*}
Based on \cref{11111} the volume function is: 
\begin{align*}
    V(x,y)=\frac{1}{(d+1)(d+2)}[D(y^{d+1})-D(x^{d+1})-D((y/x)^{d+1})]+\frac{1}{(d+2)}[D(x)-D(y)-D(x/y)];
\end{align*}
which proves the exactness of $P_d$. 
\end{proof}
\section{Toric points and sign of slopes for $P_d$}\label{Toric points}
As we have already mentioned, there is a closed formula in \cite{1} to compute the Mahler measure of exact, more precisely regular polynomials (see \cref{regular}) as follows:
\begin{align}\label{Antonin formula}
    m(P) =  \frac{1}{2\pi} \sum  \epsilon(x,y)V(x,y).
\end{align}
The summation will be on the set of toric points of $P$ (see \cref{def}); $\epsilon(x,y)$ is the opposite of the sign of the imaginary part of $\frac{x\partial_x P}{y\partial_yP}$ at toric point $(x,y)$ and $V$ is a volume function.\\
Since $P_d$ is exact, we use the formula to compute $m(P_d)$. In order to apply it we need to compute the toric points of $P_d$ and the sign of the imaginary part of $\frac{x\partial_x P_d}{y\partial_yP_d}$ at toric points. 
\begin{definition}\label{def}
The set of \textbf{toric points} of  $P \in \mathbb{C}[X,Y]$ is defined by:
$$\{ (x,y) \in {\mathbb{C}^*}^2 | P(x,y)= 0 , |x|=|y|=1\}.$$
\end{definition}
We notice that the necessary condition on $P$ for this formula to apply is that, for each toric point of $P$, the fraction  $\frac{x\partial_x P}{y\partial_yP}$, should not be real. This property leads to the definition of regular polynomials. We briefly explain some new definitions but for more information see \cite{1}.

\begin{definition}
The \textbf{logarithmic Gauss map} $\gamma : C \rightarrow \mathds{P^1}(\mathds C)$ is
defined by $\gamma (x,y) = [x\partial_xP,y\partial_yP]$.
\end{definition}
\begin{definition}\label{regular}
An exact polynomial $P(x,y)$ is called \textbf{regular} if for each toric point, $(x,y)$, we have $\gamma (x,y) \notin \mathds{P^1}(\mathds R)$.
\end{definition}
From the previous definition, $\gamma (x,y)$ is a point in projective plane. If $P$ is a regular polynomial, then in particular $y\partial_yP|_{(x,y)} \neq 0$ and $x\partial_xP|_{(x,y)} \neq 0$, and consequently, $[x\partial_xP,y\partial_yP]= [\frac{x\partial_x P}{y\partial_yP},1] \in \mathds{P^1}(\mathds C) \setminus \mathds{P^1}(\mathds R) $. Therefore, for the regular polynomial $P$, the value of $\frac{x \partial_xP}{y\partial_yP}$ at a toric point $(x,y)$, is a non real number, so we can use the mentioned formula to compute $m(P)$.
We use the two point of views in this article: $\frac{x \partial_xP}{y\partial_yP}$ and $\gamma(x,y)$. 
 \subsection{Toric points of \texorpdfstring{$P_d$}{}}
The goal of this section is to prove \cref{th3.2}; 
\begin{proposition}\label{th3.2}
The set of toric pints of $P_d(x,y)$ is as follows:

  $$\{(x,y)\in \mathds{C^*}^2\mid x^{d+1}=y^{d+1}=1,x\neq 1,y\neq 1\ , x \neq y\} \cup 
\{(x,y)\in \mathds{C^*}^2\mid x^{d+2}=y^{d+2}=1 ,x\neq 1,y\neq 1, x \neq y\}.$$

\end{proposition}
 For convenience, the first set in \cref{th3.2} is denoted by $U_{d+1}$, and the second one by $U_{d+2}$.

\begin{remark}\label{re}
If $P(x,y)\in \mathds{R}[X,Y]$ the set of toric points of $P(x,y)$ and $P^*(x,y)$ are equal, where $P^*(x,y)=P(1/x,1/y)$, with $x,y $ not equal to zero.
\end{remark}
Let $(x,y)$ be a toric point of $P_d$, using \cref{re} we have:
\begin{equation} \label{eq.1}
    P_d(x,y)=P^*_d(x,y)=0.
\end{equation}
Therefore we have $P_d(x,y)+x^{d+1}y^d P^*_d(x,y)=0$. One may check by a simple computation that we have:
 \begin{equation}\label{eq 1.4}
  P_d(x,y)+x^{d+1}y^d P^*_d(x,y)= \frac{y^{d+2}-1}{y-1} \frac{x^{d+1}-1}{x-1}. 
 \end{equation}
The previous remark leads to the following lemma;
\begin{lemma}\label{th 1.1}
The toric points of $P_d(x,y)$ are contained in:
\begin{align*}\label{set 1.5}
     \{(x,y)\in \mathds{C^*}^2\mid x^{d+1}=y^{d+1}=1,x\neq 1,y\neq 1\} \cup 
\{(x,y)\in \mathds{C^*}^2\mid x^{d+2}=y^{d+2}=1 ,x\neq 1,y\neq 1\}.
\end{align*}

\end{lemma}
\begin{proof}
If $(x,y)$ is a toric point of $P_d$, then \cref{eq.1} and \cref{eq 1.4} hold, so we have:
\begin{equation}\label{8}
 \frac{x^{d+2}-1}{x-1}=0   
\ \ \ \text{or} \ \ \  
    \frac{y^{d+1}-1}{y-1}=0. 
\end{equation}
The polynomial $P_d(x,y)$ is a symmetric polynomial, so $P_d(x,y)=P_d(y,x)$. Thus, we switch $x$ and $y$, so $(y,x)$ is a toric point as well as $(x,y)$. Hence, 
\begin{equation*}
    P_d(y,x)+y^{d+1}x^d P^*_d(y,x)=\frac{y^{d+2}-1}{y-1}\ \ \frac{x^{d+1}-1}{x-1}=0 ,
    \end{equation*}
   and we have:
    \begin{equation}\label{88}
 \frac{y^{d+2}-1}{y-1}=0   \ \ \ or \ \ 
    \frac{x^{d+1}-1}{x-1}=0.
\end{equation}
Therefore, according to \cref{8} and \cref{88} there are 4 possibilities:
\begin{enumerate}
  \item $x^{d+2}=1,x\neq 1, y^{d+2}=1 , y\neq 1$.
   \item $x^{d+2}=1, x^{d+1}=1 ,x\neq 1$, which is not compatible.
  \item $y^{d+1}=1, y^{d+2}=1 , y\neq 1$, which is not compatible.
   \item $y^{d+1}=1,y\neq 1, x^{d+1}=1 , x\neq 1$.
\end{enumerate}
  \end{proof}
  
\begin{lemma} \label{th 1.2}
If $(x,y)$ is a toric point of $P_d(x,y)$, then $x \neq y$:
\end{lemma}
\begin{proof}
Let $x$ is a $(d+1)$ or $(d+2)$ root of unity. We prove by contradiction that $P_d(x,x)$ is not equal to zero.
\begin{equation*}
   0=P_d(x,x)= \sum_{0\leq i+j \leq d} x^ {i+j}= \sum_{0\leq k \leq d}(k+1)x^{k}=  \left(\frac{d}{dx}\sum_{k=0}^{d} x^{k+1}\right) .
\end{equation*}
Therefore, $x$ is a root of $\frac{d}{dx} \left( \sum_{k=0}^{d}x^{k+1}  \right)$. The Gauss-Lucas theorem asserts that the zeroes of the derivative of a polynomial have to lie in the convex hull of the zeros of the polynomial itself. On the other side,
$$\sum_{k=0}^{d}x^{k+1}=\frac{x^{d+2}-1}{x-1}.$$
Since the two polynomials $\sum_{k=0}^{d}x^{k+1}$ and $P_d(x,x)$ are coprime to each other, $x$ is strictly inside the convex hull of $(d+2)$-roots of unity. Therefore, $|x|< 1$, 
which contradicts the fact that $x$ is a root of unity.
Hence, there is no symmetric pair $(x,x)$ in the set of toric points of $P_d$. 
\end{proof}
We are ready to prove \cref{th3.2}, which asserts that the set of toric pints of $P_d(x,y)$ is $U_{d+1} \cup U_{d+2}$;

\begin{proof}
From the two previous lemma, we know that the set of toric points of $P_d$ is included in $U_{d+1} \cup U_{d+2}$. To prove the revers we notice that for $(x,y)\in U_{d+1} \cup U_{d+2}$ we have $|x|=|y|=1$, so we just prove $P_d(x,y)=0$. To do so, we consider two cases:
\begin{itemize}
    \item Case 1) $(x,y) \in U_{d+1}=\{(x,y)\in  \mathds{C^*}^2\mid x^{d+1}=y^{d+1}=1,x\neq 1,y\neq 1\ , x \neq y\}$:
    \begin{align*}
       P_d(x,y)&=(1+x+\dots +x^d)+y(1+x+\dots +x^{d-1})+\dots +y^{d-1}(1+x)+y^d\\
       &= \frac{(x^{d+1}+ yx^{d}+\dots +y^{d-1}x^2+y^dx)-(1+y+\dots+y^d)}{x-1}.
   \end{align*} 
     Because $y$ is a $d+1$ root of unity, so $(1+y+\dots+y^d)$ is equal to zero. Also,  $0=1-1=x^{d+1}-y^{d+1}=(x-y)(x^d+x^dy+\dots +y^d)$, but $y\neq x$, so $(x^d+x^dy+\dots +y^d)=0$. Hence, $P_d(x,y)=0$. 
    
    \item Case 2) $(x,y) \in U_{d+2}=\{(x,y)\in \mathds{C^*}^2\mid x^{d+2}=y^{d+2}=1 ,x\neq 1,y\neq 1, x \neq y\}$:\\
$P_d(x,y)$, for $d\geq 1$ is symmetric, so we have:
    $$xP_d(x,y)+1+y+\dots+y^{d+1}=P_{d+1}(x,y)=P_{d+1}(y,x)=yP_d(x,y)+1+x+\dots+x^{d+1}.$$
    By subtracting $P_{d+1}(y,x)$ from $P_{d+1}(x,y)$, the following equation holds for  any toric point:
    \begin{equation}\label{e}
    (x-y)P_d(x,y)+ \frac{y^{d+2}-1}{y-1}-\frac{x^{d+2}-1}{x-1}=0.
    \end{equation}
    
  Also, $y^{d+2}-1=x^{d+2}-1=0$ and since $x \neq y$, so $P_d(x,y)=0$.
   \end{itemize}
    \end{proof}
    \begin{example}
  The following table shows the toric points of $P_2(x,y)$. 
  \begin{center}
 \begin{tabular}{ |p{2cm}||p{2cm}| }
  \hline
$(x,y) \in U_{3}$   &  $(x,y) \in U_{4}$   \\
 \hline
& $(e^{\frac{4\pi}{4}i}, e^{\frac{6\pi}{4}i})$   \\
$(e^{\frac{2\pi}{3}i},e^{\frac{4\pi}{3}i})$ &   $(e^{\frac{6\pi}{4}i},e^{\frac{4\pi}{4}i} )$  \\
&$(e^{\frac{2\pi}{4}i},e^{\frac{4\pi}{4}i})$\\
&$(e^{\frac{4\pi}{4}i},e^{\frac{2\pi}{4}i})$\\
$(e^{\frac{4\pi}{3}i},e^{\frac{2\pi}{3}i})$ &$(e^{\frac{2\pi}{4}i},e^{\frac{6\pi}{4}i}$)\\
&$(e^{\frac{6\pi}{4}i},e^{\frac{2\pi}{4}i})$\\
 \hline
\end{tabular}
  \end{center}
\end{example}
\subsection{Signs of slopes for \texorpdfstring{$P_d$}{}}
\label{Sign}
As we explained before, we need to compute $\epsilon$ at each toric point, which is the opposite of the sign of the imaginary part of $\gamma(x,y)=\frac{x\partial_x P}{y\partial_yP}$. The sign of the imaginary part of $\gamma(x,y)$ is denoted by $\sg(\IMM(\gamma(x,y)))$.

\begin{example}\label{ex 3.2} The following table shows $\sg(\IMM(\gamma(x,y)))$ at toric points of $P_2$.
\begin{center}
\begin{tabular}
{|p{2cm}|p{2.6cm}|p{2cm}||p{2cm}|p{2.6cm}|p{2cm}|}
  \hline
$(x,y) \in U_{3}$   & $\sg(\IMM(\gamma(x,y)))$& $\hspace{5mm} \epsilon(x,y)$& $(x,y) \in U_{4}$ & $\sg(\IMM(\gamma(x,y)))$ & $\hspace{5mm}\epsilon(x,y)$\\
 \hline
& & &$(e^{\frac{4\pi}{4}i}, e^{\frac{6\pi}{4}i})$ & $\hspace{1cm} -$ & $\hspace{8mm} +$ \\
$(e^{\frac{2\pi}{3}i},e^{\frac{4\pi}{3}i})$ & $\hspace{1cm} +$ & $\hspace{8mm} -$& $(e^{\frac{6\pi}{4}i},e^{\frac{4\pi}{4}i} )$&$\hspace{1cm} +$ &$\hspace{8mm} -$ \\
& & &$(e^{\frac{2\pi}{4}i},e^{\frac{4\pi}{4}i})$& $\hspace{1cm} -$& $\hspace{8mm} +$\\
& & &$(e^{\frac{4\pi}{4}i},e^{\frac{2\pi}{4}i})$&$\hspace{1cm} +$&$\hspace{8mm} -$\\
$(e^{\frac{4\pi}{3}i},e^{\frac{2\pi}{3}i})$ &$\hspace{1cm} -$&$\hspace{8mm} +$ &$(e^{\frac{2\pi}{4}i},e^{\frac{6\pi}{4}i}$)&$\hspace{1cm} -$&$\hspace{8mm} +$ \\
& & &$(e^{\frac{6\pi}{4}i},e^{\frac{2\pi}{4}i})$& $\hspace{1cm} +$&$\hspace{8mm} -$\\
 \hline
\end{tabular}
\end{center}
   \end{example}
 To generalize the above table for any $P_d$ we define $\Omega$, which associates each toric point with a point in ${\mathds{R}}^2$. The map is defined by 
   $\Omega:(x_i,y_i) \mapsto(l_i,k_i)$,  where  $(x_i,y_i)=(\omega^{l_i},\omega^{k_i})$, and
$\omega= e^{\frac{2 \pi }{d+1}i}$ if $(x_i,y_i) \in U_{d+1}$, or $\omega= e^{\frac{2 \pi }{d+2}i}$ if $(x_i,y_i) \in U_{d+2}$. We say $\Omega:(x,y) \mapsto(l,k)$ is above the diagonal if $l<k$ and below the diagonal if $k<l$. Note that by \cref{th 1.2}, $l \neq k$;
 Now, we prove the following proposition for the sign of slopes.
\begin{proposition}\label{th4.1}
Let $d\geq 1$, for the polynomial $P_d(x,y)$,  $\sg(\IMM(\gamma(x,y)))$ at each toric point is determined as follows;
\begin{itemize}
    \item For $(x,y)\in U_{d+1}$:
     \begin{itemize}
        \item If $\Omega(x,y)$ is above the diagonal, the sign is positive, so $\epsilon(x,y)<0$.
        \item If $\Omega(x,y)$ is below the diagonal, the sign is  negative, so $\epsilon(x,y)>0$.
        \end{itemize}
    \item For $(x,y)\in U_{d+2}$:
     \begin{itemize}
        \item If $\Omega(x,y)$ is above the diagonal, the sign is negative, so $\epsilon(x,y)>0$.
        \item If $\Omega(x,y)$ is below the diagonal, the sign is positive, so $\epsilon(x,y)<0$.    
    \end{itemize}
\end{itemize}
\end{proposition}
\begin{proof}
We find $\sg(\IMM(\gamma(x,y)))$. Recall that $\epsilon(x,y)$ is its opposite!
As we saw in the proof of \cref{th3.2}, at a toric point $(x,y)$ \cref{e} is satisfied:
$$0=(x-y)P_d(x,y)+\frac{y^{d+2}-1}{y-1}-\frac{x^{d+2}-1}{x-1}.$$
Let $Q(x,y)=(x-1)(y-1)(x-y)$. For all $(x,y) \in  {\mathds{C}}^2$ we have this equality of polynomials:
$$P_d (x,y)Q(x,y) =(x^{d+2}-1)(y-1)-(y^{d+2}-1)(x-1).$$
We apply $\partial_x$ and $\partial_y$ to the both sides of the above equality:
\begin{equation} \label{eq 2.1}
    \partial_xP_d (x,y)Q(x,y)+\partial_xQ(x,y)P_d (x,y)=(d+2)(y-1)x^{d+1}-(y^{d+2}-1),
\end{equation}
\begin{equation} \label{eq 2.2}
    \partial_yP_d (x,y)Q(x,y)+\partial_yQ(x,y)P_d (x,y)=(x^{d+2}-1)-(d+2)(x-1)y^{d+1}.
\end{equation}
We divide \cref{eq 2.1} by \cref{eq 2.2}, so for all the $(x,y) \in  {\mathds{C}}^2$ we have:

\begin{equation} \label{eq 2.3}
    \frac{\partial_xP_d (x,y)Q(x,y)+\partial_xQ(x,y)P_d (x,y)}{ \partial_yP_d (x,y)Q(x,y)+\partial_yQ(x,y)P_d (x,y)}=\frac{(d+2)(y-1)x^{d+1}-(y^{d+2}-1)}{(x^{d+2}-1)-(d+2)(x-1)y^{d+1}}.
\end{equation}
We evaluate the previous equation at toric points and we consider two cases:
\begin{itemize}
    \item Case 1) $(x,y)\in U_{d+1}$:
$$\frac{\partial_xP_d (x,y)}{\partial_yP_d (x,y)}=-\frac{y-1}{x-1}\ , \ \text{so}\ \  \frac{x\partial_xP_d (x,y)}{y\partial_yP_d (x,y)}=\frac{-x(1-y)}{y(1-x)}.$$
  
    \item  Case 2) $(x,y) \in U_{d+2}$:
    $$ \frac{\partial_xP_d (x,y)}{\partial_yP_d (x,y)}=-\frac{x^{d+1}(y-1)}{y^{d+1}(x-1)}\  ,\ \text{so}\ \  \frac{x\partial_xP_d (x,y)}{y\partial_yP_d (x,y)}=-\frac{1-y}{1-x}.$$ 
       \end{itemize}
To compute $Sgn(\IMM(\gamma (x,y)))$ at toric points, we consider both $x$ and $y$ as a suitable power of the associated first primitive root of unity. 
     \begin{figure}[!htb]
   \begin{minipage}{0.48\textwidth}
     \centering
     \includegraphics[width=0.6\linewidth]{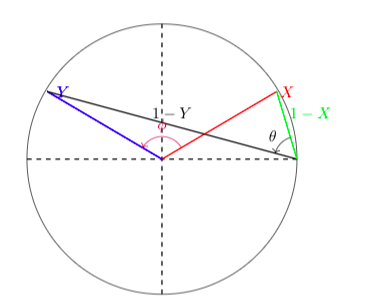}
     \caption{}\label{fig 1teta}
   \end{minipage}\hfill
   \begin{minipage}{0.48\textwidth}
     \centering
     \includegraphics[width=.6\linewidth]{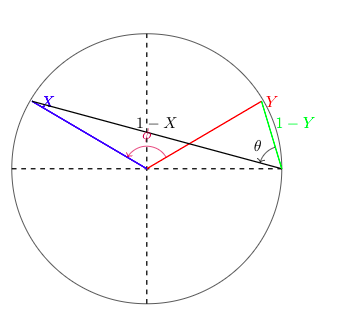}
     \caption{}\label{fig 2teta}
   \end{minipage}
\end{figure}
     \begin{itemize}
    \item Case 1) $(x,y)\in U_{d+1}$:
    Let $\omega=e^{\frac{2\pi}{d+1}i}$. There are some $0 < a \leq d$ and $0 < b \leq d$, with $x=\omega^a, y=\omega^b$ and $a \neq b$. We consider $\Omega(x,y)$ and we have two possible cases:
    
   \begin{enumerate}
       \item If $\Omega(x,y)$ is above the diagonal, or equivalently $b>a$ (see \cref{fig 1teta}), we have: 
   \begin{equation} \label{eq 2.4}
        \gamma(x,y)=\frac{x\partial_xP_d (x,y)}{y\partial_yP_d (x,y)}= \frac{-x}{y} \ \frac{1-y}{1-x} = -e^{-i\phi}r e^{i\theta}.
    \end{equation}
    In the last equality in \cref{eq 2.4}, we used the suitable polar representations according to \cref{fig 1teta}, where $\frac{x}{y}=e^{-i\phi}$, with $0 < \phi <2\pi $ and $\frac{1-y}{1-x}=r e^{i\theta}$, with $r>0$, $0 < \theta < \pi$. We notice that $\phi$ and $\theta$ are respectively central and inscribed angles with the same intercepted arc in the circle, so $\phi = 2 \theta$. Therefore, we have:
     $$\sg(\IMM( \frac{x\partial_xP_d (x,y)}{y\partial_yP_d (x,y)}))= - \sg(\IMM(re^{-i\phi/2}))= \sin(\frac{\phi}{2}),$$
     since $0 < \phi < 2\pi $, $\sg(\IMM(\gamma (x,y)))$ is positive.

  \item If $\Omega(x,y)$ is below the diagonal, or equivalently $a>b$ (see \cref{fig 2teta}), we have:
  $$\sg(\IMM( \frac{x\partial_xP_d (x,y)}{y\partial_yP_d (x,y)}))= - \sg(\IMM(e^{i\phi/2}))= -\sin(\frac{\phi}{2}),$$
  so, $\sg(\IMM(\gamma (x,y)))$ is negative.
  \end{enumerate}
    \vspace{.3cm}
    \item Case 2) $(x,y)\in U_{d+2}$: Let $\omega=e^{\frac{2\pi}{d+2}i}$ and there are some $0 < a \leq d+1$ and $0 < b \leq d+1$ such that, $x=\omega^a, y=\omega^b$ and $a \neq b$. Again we check the two possible cases:
    \begin{enumerate}
    \vspace{.3cm}
       \item If $\Omega(x,y)$ is above the diagonal, or equivalently $b>a$ (see \cref{fig 1teta}), we have:
 
    $$\sg(\IMM(\frac{x\partial_xP_d (x,y)}{y\partial_yP_d (x,y)}))=\sg(\IMM(-\frac{1-y}{1-x}))= \sg(\IMM(-re^{i \theta}))= -\sin(\theta)=-\sin(\frac{\phi}{2}),$$
           
     since $0 <\phi< 2\pi$, $\sg(\IMM(\gamma(x,y)))$ is negative.
       
       \item If $\Omega(x,y)$ is below the diagonal, or equivalently $a>b$ (see \cref{fig 2teta}), we have:
    $$\sg(\IMM(\frac{x\partial_xP_d (x,y)}{y\partial_yP_d (x,y)}))=\sg(\IMM(-\frac{1-y}{1-x}))= \sg(\IMM(-re^{-i \theta}))= -\sin(-\theta)=\sin(\frac{\phi}{2}),$$
      
       thus $\sg(\IMM(\gamma(x,y)))$ is positive.
       
     \end{enumerate}
   \end{itemize}
   
\end{proof}
An immediate result from the previous proposition is that $P_d$ is regular, for $d \geq 1$, since for a toric point $(x,y)\in U_{d+1}\cup U_{d+2}$ we have:
\begin{align*}
     \sg(\IMM(\frac{x\partial_xP_d (x,y)}{y\partial_yP_d (x,y)})) \in \mathbb{R} \Leftrightarrow \pm \sin(\frac{\phi}{2})= 0 \Leftrightarrow \frac{\phi}{2}=2k\pi.
    \end{align*}
    As we consider $0 < \phi < 2\pi$, so it is impossible. 
 \subsection{Computing the Mahler measures of $P_1$ and $P_2$}
 \label{Mahler}
We may use the formula for Mahler measures to compute $m(P_d)$, for arbitrary values of $d$. Let us do it explicitly for $d=1,2$. We write thus $m(P_1)$ and $m(P_2)$ as a finite sums of the values of $D$, at specific roots of unity. The case of $P_1$ was first computed by Smyth  \cite{smyth_1981}.
\begin{observation}
We have $m(P_1)=\frac{1}{\pi}D(e^{\frac{\pi}{3}i})$, which is approximately $0.32$. 
\end{observation}
By using \cref{th3.2} the set of the toric points of $P_1$ is $U_2 \cup U_3= U_3$. \cref{th4.1} gives the following values for $\epsilon(x,y)$:
\begin{center}
 \begin{tabular}{ |p{2cm}||p{2cm}| }
  \hline
$(x,y) \in U_{3}$   &  $\hspace{4mm}\epsilon(x,y)$   \\
 \hline
$(e^{\frac{2\pi}{3}i},e^{\frac{4\pi}{3}i})$ & \hspace{8mm}$+$  \\
$(e^{\frac{4\pi}{3}i},e^{\frac{2\pi}{3}i})$ & \hspace{8mm}$-$ \\
 \hline
\end{tabular}
  \end{center}
According to \cref{th 2.2}, $V(x,y)=-D(-x)$ is a volume function for $P_1$. We notice that for $\omega$ on the unite circle we have $D(\bar{\omega})=-D(\omega)$, $D(-e^{i\phi })=D(e^{i(\pi+\phi) })$, and $D(\omega e^{i2\pi })=D(\omega)$. Hence, we have:
\begin{align*}
    2\pi m(P_1)&= \sum_{(x,y) \in U_3} \epsilon(x,y)V(x,y)=\epsilon(e^{\frac{i2\pi}{3}},e^{i\frac{4\pi}{3}})(-D(-e^{i\frac{2\pi}{3}}))+ \epsilon(e^{i\frac{4\pi}{3}},e^{i\frac{2\pi}{3}})(-D(-e^{i\frac{4\pi}{3}}))\\
    &=(-D(-e^{i\frac{2\pi}{3}}))-(-D(-e^{i\frac{4\pi}{3}}))=-D(e^{i(\pi+\frac{2\pi}{3})})+D(e^{i(\pi+\frac{4\pi}{3})})=-D(e^{i\frac{5\pi}{3}})+D(e^{i\frac{7\pi}{3}})\\
    &= D(\overline{e^{i\frac{5\pi}{3}}})+D(e^{i\frac{\pi}{3}})=D(e^{i\frac{\pi}{3}})+D(e^{i\frac{\pi}{3}})=2D(e^{i\frac{\pi}{3}}).
\end{align*}
Therefore, $m(P_1)=\frac{1}{\pi} D(e^{i\frac{\pi}{3}})$.

\begin{observation}
We have $m(P_2)=\frac{1}{2 \pi}\big(\frac{3}{2}D(e^{i\frac{4\pi}{3}})+4D(e^{i\frac{\pi}{2}})\big)$, which is approximately $0.421$.
\end{observation}
Notice that $U_3 \cup U_4$ is the set of toric points of $P_2$. We have $\epsilon$ at each toric point by looking at the table in \cref{ex 3.2}.  According to \cref{th2.2}, a volume function is
$V(x,y)=\frac{1}{12} (D(y^3)-D(x^3)-D(y/x)^3)+\frac{1}{4} (D(x)-D(y)-D(x/y))$. The value of the volume function at $(x,y)$ in $U_3$ is equal to $\frac{1}{4}(D(x)-D(y)-D(x/y))$ and, for $(x,y)$ in $U_4$ is equal to $\frac{1}{3}(D(x)-D(y)-D(x/y))$. Hence, we have:
\begin{align*}
 2\pi m(P_2) &= \frac{1}{4}\sum_{(x,y) \in U_3} \epsilon(x,y)(D(x)-D(y)-D(x/y))+\frac{1}{3}\sum_{(x,y) \in U_4} \epsilon(x,y)(D(x)-D(y)-D(x/y)).
\end{align*}
For any $(x,y) \in U_3$ we have $(x,y)=(x,\bar{x})$. In addition, $x$ and $y$ are third roots of unity, so that $D(x/y)=D(y)$, and thus we have $D(x)-D(y)-D(x/y)=3D(x)$. Hence, the first summation in $m(P_2)$ is $\frac{3}{4}\sum_{(x,y) \in U_3} \epsilon(x,y) D(x)$, which is $\frac{3}{4}(D(e^{i\frac{4\pi}{3}})-D({e^{i\frac{2\pi}{3}}}))=\frac{3}{2}D(e^{i\frac{4\pi}{3}})$. We now, look at the second summation in $m(P_2)$. From \cref{ex 3.2}, for any $(x,y) \in U_4$ we have $(y,x) \in U_4$, and $\epsilon(x,y)=-\epsilon(y,x)$. Therefore, we can rewrite the second summation as:
\begin{align*}
    \frac{1}{3}\sum_{(x,y) \in U_4} \epsilon(x,y)V(x,y)=\frac{2}{3}\sum_{(x,y) \in U_4,\  \epsilon(x,y)>0 } V(x,y)= \frac{2}{3}\sum_{(x,y) \in U_4,\  \epsilon(x,y)>0 } (D(x)-D(y)-D(x/y)).
\end{align*}
The points $(x,y) \in U_4$ with $\epsilon(x,y)>0$ are the points  $(e^{i\frac{\pi}{2}},e^{i\pi})$, $(e^{i\pi},e^{i\frac{3\pi}{2}})$ and $(e^{i\frac{\pi}{2}},e^{i\frac{3\pi}{2}})$. To simplify the calculation, $e^{i\frac{\pi}{2}}$ , $e^{i\pi}$, and $e^{i\frac{3\pi}{2}}$ are respectively denoted by $a,b$ and $c$ and we have:
\begin{align*}
    &\frac{2}{3}\sum_{(x,y) \in U_4,\  \epsilon(x,y)>0 } (D(x)-D(y)-D(x/y))\\
    &=\frac{2}{3}\big(D(a)-D(b)-D(a/b)+D(b)-D(c)-D(b/c)+D(a)-D(c)-D(a/c)\big)\\
    &=\frac{2}{3}\big(2D(a)-2D(c)-D(a/b)-D(b/c)-D(a/c)\big)\\
    &=\frac{2}{3}\big(2D(e^i\frac{\pi }{2})-2D(e^{i\frac{3\pi}{2}})-D(e^{-i\frac{\pi}{2}})-D(e^{-i\frac{\pi}{2}})-D(e^{-i\pi}) \big)\\
    &=\frac{2}{3}\big(2D(e^{i\frac{\pi}{2}})-2D(e^{i\frac{3\pi}{2}}) +2D(e^{i\frac{\pi}{2}})\big)=\frac{2}{3}\big(6D(e^{i\frac{\pi}{2}})\big)=4D(e^{i\frac{\pi}{2}})).
    \end{align*}
  Note that $D(e^{-i\frac{\pi}{2}})=-D(e^{i\frac{\pi}{2}})$, $D(e^{i\frac{3\pi}{2}})=-D(e^{i\frac{\pi}{2}})$, and $D(e^{-i \pi })=0$ . In the end, we find the evaluation of $m(P_2)$ as follows:
  $$ m(P_2)=\frac{1}{2 \pi}\big(\frac{3}{2}D(e^{i\frac{4\pi}{3}})+4D(e^{i\frac{\pi}{2}})\big).$$
  \section{Experimental computations}
  The previous computation may be automated. We get an algorithm to compute the Mahler measure of any $P_d$, as a combination of dilogarithm at roots of unity. This can be computed with arbitrary precision in a very efficient way.
 For $1\leq d \leq 1000$ the graph of $m(P_d)$, implemented in SageMath, is shown in \cref{fig: MM}.\\
\begin{figure}[!htpb]
    \centering
    \includegraphics[width=.6 \textwidth]{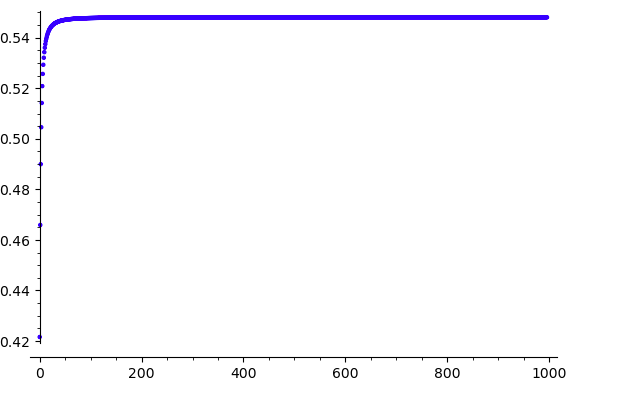}
    \caption{The graph of $m(P_d)$, for $1\leq d \leq 1000.$}
    \label{fig: MM}
\end{figure}

The figure hints to the existence of a limit for $m(P_d)$. We study more about the volume function to prove the existence  of the limit. 
\section{Volume function at toric points}\label{6}
\subsection{Simplification of volume function at toric points}
In the formula for $m(P_d)$, we may simplify the computation of the volume function at toric points. Indeed, the values of this function at $(x,y)\in U_{d+1}$ may be rewritten as follows:
 \begin{align*}
    V(x,y)= & \frac{1}{(d+1)(d+2)}[D(y^{d+1})-D(x^{d+1})-D((y/x)^{d+1})] +\frac{1}{(d+2)}[D(x)-D(y)-D(x/y)]\\
    = & \frac{1}{(d+2)}[D(x)-D(y)-D(x/y)].
\end{align*} 
Likewise, at a point $(x,y)\in U_{d+2}$ we have:
\begin{align*}
    V(x,y)= &
    \frac{1}{(d+1)}[D(x)-D(y)-D(x/y)].
\end{align*}
In the rest of this section, by simplifying the volume function at toric points we introduce a new function, called $\vol$ . 

\begin{definition}\label{pro 5.1}
The function
 $\vol:[0,2\pi]\times [0,2\pi]\mapsto \mathds{R}$, defined by
 $\vol(\theta,\alpha):=D(e^{ i \theta })-D(e^{i(\theta+\alpha)})+D(e^{i\alpha})$ has the following properties:
\begin{enumerate}
    \item For $(x,y)\in U_{d+1}$, with $x=e^{\frac{2k\pi i}{d+1}}$, $y=e^{\frac{2k'\pi i}{d+1}}$, where $0 <k<k'< d+1$ we have:
    \normalfont
     \begin{align*}
      V(e^{\frac{2k\pi i}{d+1}}, e^{\frac{2k'\pi i}{d+1}})
      = \frac{1}{d+2}\vol\left(\frac{2k\pi}{d+1},\frac{2(k'-k)\pi}{d+1}\right)=-V(e^{\frac{2k'\pi i}{d+1}}, e^{\frac{2k\pi i}{d+1}}).
   \end{align*}
   \itshape
 \item For $(x,y) \in U_{d+2}$, with $x=e^{\frac{2k\pi i}{d+2}}$ and $y=e^{\frac{2k'\pi i}{d+2}}$, where $0 <k<k'< d+2$ we have:
 \normalfont
 \begin{align*}
     V(e^{\frac{2k\pi i}{d+2}}, e^{\frac{2k'\pi i}{d+2}})=
     \frac{1}{d+1} \vol\left(\frac{2k\pi}{d+2},\frac{2(k'-k)\pi}{d+2}\right)= -V(e^{\frac{2k'\pi i}{d+2}}, e^{\frac{2k\pi i}{d+2}}).
 \end{align*}
\end{enumerate}
\end{definition} 
\begin{notation}
The triangle with vertices $\{(0,0),(0,2\pi),(2\pi,0)\}$, is denoted by $T$.
\end{notation}
The following lemma states another important property of $\vol$;
\begin{lemma}\label{lem 7.1}

The function, $\vol(\theta,\alpha)$,
 is positive inside of $T$ and equals zero on its boundary.
\end{lemma}
\begin{proof}
 $\vol$ is continuous everywhere and real analytic everywhere except at $(\theta,\alpha)$ where $e^{i\theta}=1 $, $e^{i\alpha} =1$  or $e^{i(\theta+\alpha)}=1 $. Each boundary point of $T$, satisfies one of the following conditions:
\begin{enumerate}
\item The point $(\theta,\alpha)$ is on $\theta=0$.  Hence, $\vol(0,\alpha)=  D(e^{i0})-D(e^{i(0+\alpha)})+D(e^{i\alpha})=D(1)+D(e^{i\alpha})-D(e^{i\alpha})=0$.
\item The point $(\theta,\alpha)$ is on $\alpha=0$, so again $\vol(\theta,\alpha)=0$.
\item The point $(\theta,\alpha)$ is on $\theta+\alpha=2\pi$. Hence, $\vol(\theta,\alpha)=D(e^{i\theta})-D(e^{i(2\pi)})+D(e^{i(2\pi-\theta)})=D(e^{i\theta})-D(e^{i\theta})=0$. Notice that  $D(\bar{z})=-D(z)$.
\end{enumerate}
Therefore, $\vol$ is equal to zero at boundary points of $T$.
 Thus, we check the sign of $\vol$, at inner points of $T$, where the function is differentiable. To do so, first, we find the critical points of $\vol$. Hence, we search for $(\theta_0,\alpha_0)$, which satisfies the following:
   $$\frac{\partial \vol}{\partial \theta}|_{(\theta_0,\alpha_0)}=\frac{\partial \vol}{\partial \alpha}|_{(\theta_0,\alpha_0)}=0.$$
   To solve the above differential system of equations, first, we compute $\frac{\partial \vol}{\partial \theta}$: 
 \begin{align*}
   \frac{\partial \vol}{\partial \theta}
   =\frac{\partial  D(e^{i\theta})}{\partial \theta}- \frac{\partial D(e^{i(\theta+\alpha)})}{\partial \theta}.
\end{align*}
We compute $\frac{\partial  D(e^{i\theta})}{\partial \theta}$, using the fact that $-dD(z)= \eta_{(z,1-z)}$ or equivalently $dD(z)=\eta_{(1-z,z)}$.
Let
$Z(\theta) =e^{i \theta}$ 
and $z_0=Z(\theta_0)= e^{i \theta_0}$:

\begin{equation*}
   \frac{\partial D(e^{i\theta})}{\partial \theta}|_{(\theta_0,\alpha_0)}=d D|_{z_0} (\frac{d}{d\theta} e^{i \theta}|_{\theta_0})=
 \eta_{(1-z_0,z_0)}  (\frac{d}{d\theta} e^{i \theta}|_{\theta_0})
 = -\log |1-e^{i\theta_0}|\ \big(d \arg_{z_0}(\frac{d}{d\theta} e^{i \theta}|_{\theta_0})\big)
 \end{equation*}
 \begin{equation*}
    -\log |1-e^{i\theta_0}|\big(\frac{d}{d\theta}\arg (e^{i\theta})|_{\theta_0} \big)= -\log |1-e^{i\theta_0}|\big(\frac{d}{d\theta} \theta |_{\theta_0}\big)
 = -\log |1-e^{i\theta_0}|\big(1|_{\theta_0}\big) = -\log |1-e^{i\theta_0}|.
    \end{equation*} 
    
    In the same way, we compute the rest of the partial derivatives. We have: 
   \begin{align*}
    \frac{\partial D(e^{i\alpha})}{\partial \alpha}  =
 -\log |1-e^{i\alpha}|\ \ \ \ \ 
    \frac{\partial D(e^{i(\theta+\alpha)})}{\partial \alpha}=\frac{\partial D(e^{i(\theta+\alpha)})}{\partial \theta}  =
 -\log |1-e^{i(\theta+\alpha)}|.
 \end{align*}
Thus, the critical points are obtained by solving the following:
 \begin{align*}
     \frac{\partial \vol}{\partial \theta}=\log |1-e^{i (\theta+\alpha)}| -\log |1-e^{i \theta}|=
\frac{\partial \vol}{\partial \alpha}= \log |1-e^{i (\theta+ \alpha)}|-\log  |1-e^{i \alpha}|=0.  \end{align*}
Therefore, we have:
 \begin{align*}
     \log |1-e^{i (\theta+ \alpha)}|-\log |1-e^{i \alpha}| = \log |1-e^{i (\theta+ \alpha)}|-\log |1-e^{i \theta}|=0. 
   \end{align*}
 We assume that $0 < \theta< 2 \pi$ , $0 < \alpha < 2 \pi$ and $0 < \alpha + \theta < 2 \pi$, since we search for the solutions of the system inside $T$. Hence, the unique critical point correspond to $\theta=\alpha=2 \pi /3$. Note that $\vol(2 \pi /3,2 \pi /3)=3D(e^{\frac{2 \pi } {3}i})$ is approximately $2,03$. Now we continue the proof by contradiction.\\
Suppose $(\theta_0, \alpha_0)\in T^{\circ}$, with $\vol(\theta_0, \alpha_0)<0$. Therefore, there exists a minimum   denoted by $(\theta_1,\alpha_1)$ where $\vol(\theta_1,\alpha_1)<0$. Note that $\vol$ is differentiable inside $T$, so the minimum is another critical point inside $T$, which is different from $(2 \pi /3,2 \pi /3)$, but this is a contradiction. Hence, $\vol$ is positive inside $T$.
\end{proof}
\subsection{Concavity of $\vol$ on $T$}
In the previous section, $\vol$ was defined and we proved that is non-negative on $T$. In this section, we prove it is concave on $T$, which will be the main tool for finding $\lim_{d\rightarrow \infty} m(P_d)$. \cref{fig MM} is the graph of $\vol$.
\begin{figure}[!htpb]
    \centering
    \includegraphics[width=.25 \textwidth]{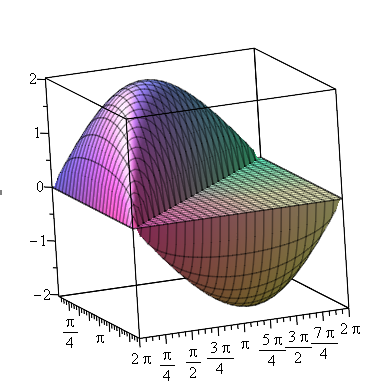}
    \caption{The graph of $\vol(\theta,\alpha).$}
    \label{fig MM}
\end{figure}
\begin{proposition}\label{pro 8.1}
 The function $\vol(\theta,\alpha)$
 is concave on $T$.
 \end{proposition}
 \begin{proof}
   We compute the Hessian matrix of $\vol$, then we prove it is negative definite. To do so, we only compute $\frac{\partial \log |1-e^{i \theta}|}{ \partial \theta}$ and the rest is done in the same way:\\
  \begin{align*}
      \frac{\partial\vol}{\partial \theta}=\log |1-e^{i (\theta+\alpha)}| -\log |1-e^{i \theta}|,\\
      \frac{\partial \vol}{\partial \alpha}=  \log | 1-e^{i(\theta+ \alpha)}|-\log |1-e^{i \alpha}|,\\
      \frac{\partial^2 \vol}{\partial \theta^2}=  \frac{\partial \log | 1-e^{i(\theta+ \alpha)}|}{\partial \theta}-\frac{\partial \log |1-e^{i \theta}|}{ \partial \theta}. 
  \end{align*}
  
We have $0\leq \theta \leq 2\pi$, so $|1-e^{i \theta}|=2 \sin (\theta/2)$.
  Therefore, we have $\frac{\partial \log |1-e^{i \theta}|}{ \partial \theta}=\frac{\partial \log (2 \sin (\theta/2)) }{ \partial \theta}=\frac{1}{2} \cot (\frac{\theta}{2})$.
After computing all the partial derivatives the Hessian matrix of $\vol$ is:

$$\mathbf H = 
 \left[ {\begin{array}{cc}
   \frac{\partial^2 \vol}{\partial \theta^2} & \frac{\partial^2 \vol}{\partial \theta \partial  \alpha} \\
   \\
   \frac{\partial^2 \vol}{\partial \alpha \partial  \theta} & \frac{\partial^2 \vol}{\partial \alpha^2} \\
  \end{array} } \right]  =
  \left[ {\begin{array}{cc}
    \frac{1}{2} \cot (\frac {\theta+\alpha}{2})-\frac{1}{2} \cot (\frac{\theta}{2}) &  \frac{1}{2} \cot (\frac {\theta+\alpha}{2}) \\
   \\
    \frac{1}{2} \cot (\frac {\theta+\alpha}{2}) &  \frac{1}{2} \cot (\frac {\theta+\alpha}{2})-\frac{1}{2} \cot (\frac{\alpha}{2}) \\
  \end{array} } \right].
$$
 The symmetric $(2 \times 2)$ Hessian matrix is  negative definite if and only if $D_1<0$ and $D_2>0$, where $D_i, (i=1,2)$ are leading principal minors. Then we compute the minors (inside $T$).
 \begin{itemize}
     \item Computation of $D_1$: $D_1= \frac{1}{2} \cot (\frac {\theta+\alpha}{2})-\frac{1}{2} \cot (\frac{\theta}{2}).$
     
The $\cot$ is decreasing between $[0,\pi]$ and since  $\alpha > 0$, we have $\frac{\theta}{2} ,  \frac{\theta+\alpha}{2}\in [0,\pi]$. Hence, 
$\cot (\frac {\theta+\alpha}{2}) < \cot (\frac {\theta}{2})$, and $D_1<0$.

\item Computation of $D_2$: 
 \begin{align*}
      & D_2 = Det(\mathbf H)\\
      & = \frac{1}{2}\left( \cot^2 (\frac {\theta+\alpha}{2})- \cot (\frac{\theta}{2}) \cot (\frac {\theta+\alpha}{2}) - \cot (\frac {\theta+\alpha}{2})\cot (\frac{\alpha}{2})  + \cot (\frac{\theta}{2})cot (\frac{\alpha}{2})- \cot^2 (\frac {\theta+\alpha}{2})\right)\\
      & = \frac{1}{2}\left(\cot (\frac{\theta}{2}) \cot (\frac{\alpha}{2})-\frac{\cot (\frac {\theta}{2})\cot(\frac{\alpha}{2})-1}{(\cot (\frac{\theta}{2})+ \cot (\frac{\alpha}{2}))}.(\cot (\frac{\theta}{2})+ \cot (\frac{\alpha}{2}))\right)=\frac{1}{2}.
  \end{align*}
Therefore, we have $D_2=\frac{1}{2}>0$ and consequently $\vol(\theta,\alpha)$ is concave inside $T$.
\end{itemize}
\end{proof}

\section{Convergence of $m(P_d)$}
\label{7}
In this section, using the relation between the values of the volume function at toric points and $\vol$, we compute $m(P_d)$ in terms of $\vol$. This computation leads to writing $m(P_d)$ as a difference of two expressions, and each of them is proportional to a Riemann sum of $\vol$ over $T$. After computing $\iint \limits_{T}\vol(\theta , \alpha) dA$, in \cref{52} and \cref{53}, we bound the errors between Riemann sums and this integral. At the end we use this bound to find the limit of $m(P_d)$.
\subsection{Computing\texorpdfstring{$\iint \limits_{T}\vol(\theta , \alpha) dA$}{Text}}
First of all, we recompute $m(P_d)$ in terms of the sum of the values of $\vol$;
\begin{theorem}\label{lem 9.1}
We have:
\normalfont
   $$ 2\pi m(P_d)= \frac{-2}{d+2} \sum_{0<k<k'\leq d} \vol\left(\frac{2k\pi}{d+1},\frac{2(k'-k)\pi}{d+1}\right)+\frac{2}{d+1} \sum_{0<k<k'\leq d+1} \vol\left(\frac{2k\pi}{d+2},\frac{2(k'-k)\pi}{d+2}\right).$$

\end{theorem} 
\begin{proof}
We use the formula for the Mahler measure \cite{1};
$$m(P_d) =  \frac{1}{2\pi} \sum_{(x,y) \in U_{d+1} \cup \  U_{d+2}} \epsilon(x,y)V(x,y).$$
We break the sum into the two summations over $d+1$, and $d+2$ toric points. \cref{th4.1} gives the value of $\epsilon(x,y)$ at each toric point. Using \cref{pro 5.1} we have:
\begin{align}\label{formul}
    2\pi m(P_d) & = \frac{-1}{d+2} \sum_{0<k<k'\leq d} \vol\left(\frac{2k\pi}{d+1},\frac{2(k'-k)\pi}{d+1}\right)
    -  \frac{1}{d+2} \sum_{0<k<k'\leq d} \vol\left(\frac{2k\pi}{d+1},\frac{2(k'-k)\pi}{d+1}\right)\\
    & + \frac{1}{d+1} \sum_{0<k<k'\leq d+1} \vol\left(\frac{2k\pi}{d+2},\frac{2(k'-k)\pi}{d+2}\right)
    + \frac{1}{d+1} \sum_{0<k<k'\leq d+1} \vol\left(\frac{2k\pi}{d+2},\frac{2(k'-k)\pi}{d+2}\right)
     \\
    & = \frac{-2}{d+2} \sum_{0<k<k'\leq d} \vol\left(\frac{2k\pi}{d+1},\frac{2(k'-k)\pi}{d+1}\right)+\frac{2}{d+1} \sum_{0<k<k'\leq d+1} \vol\left(\frac{2k\pi}{d+2},\frac{2(k'-k)\pi}{d+2}\right).
\end{align}
\end{proof}

In \cref{formul}, when $d$ goes to infinity each summation looks like a Riemann sum of $\vol$ over $T$. We compute $\iint \limits_{T} \vol(\theta , \alpha) dA $, where $dA$ is the euclidean measure on $T$. 

\begin{lemma} \label{lem 9.2}
We have:
\normalfont
$$\iint \limits_{T} \vol(\theta , \alpha) dA =6\pi \zeta (3).$$
\end{lemma}
\begin{proof}
In this proof, we use the formula, $D(e^{ i \theta})=\sum_{n=1}^{\infty} \frac{\sin (n\theta)}{n^2}$ (see \cref{def 2.3}). The summation and the integration commute, since the series converges uniformly;
\begin{align*}
         \iint \limits_{T} \vol(\theta , \alpha) dA &= \int_{0}^{2\pi}\int_{0}^{2\pi-\alpha} \vol(\theta , \alpha) d\theta d \alpha\\
         &=\int_{0}^{2\pi}\int_{0}^{2\pi-\alpha} D(e^{ i \theta })-D(e^{i(\theta+\alpha)})+D(e^{i\alpha})d\theta d \alpha
         \\
         &=\int_{0}^{2\pi}\int_{0}^{2\pi-\alpha} \big(\sum_{n=1}^{\infty} \frac{\sin (n\theta)}{n^2}+\sum_{n=1}^{\infty} \frac{\sin (n\alpha)}{n^2}-\sum_{n=1}^{\infty} \frac{\sin (n(\theta+\alpha))}{n^2}\big)d\theta d \alpha\\
         &=\sum_{n=1}^{\infty}\int_{0}^{2\pi}\int_{0}^{2\pi-\alpha} \frac{\sin (n\theta)+\sin (n\alpha)-\sin (n(\theta+\alpha))}{n^2}d\theta d \alpha\\
         &=\sum_{n=1}^{\infty}\int_{0}^{2\pi} \big[\frac{\cos (n(\theta+\alpha))-\cos (n\theta)+n\theta \sin (n\alpha)}{n^3}\big]_{0}^{2\pi-\alpha}d \alpha\\
         &=\sum_{n=1}^{\infty}\int_{0}^{2\pi} \frac{2-2\cos(n\alpha)+n(2\pi - \alpha)\sin (n \alpha)}{n^3}d \alpha\\
         &=6\pi \sum_{n=1}^{\infty} \frac{1}{n^3}=6\pi \zeta (3).
   \end{align*}
\end{proof}
\subsection{An upper bound for the integral} \label{52}
 In this section, we exhibit an upper bound for the integral of $\vol(\theta,\alpha)$ using affine functions. First, we introduce a subpartition of $T$, and we define a summation over this subpartition. In \cref{con 9.2}, using the fact that any tangent plane to the graph of a concave function is above the graph, we find our upper bound.
\begin{observation} \label{ob 9.1}
\textbf{Square subpartition}:\\ Consider the set of the points $(\frac{2k\pi}{d+1},\frac{2(k'-k)\pi}{d+1})$ with $0<k<k'<d+1$ inside $T$. For $(x,y)$ in the set, consider the square with side $\frac{ 2\pi}{d+1}$ such that $(x,y)$ is at the center of the square. The union of the squares is called \textbf{(d+1)-square subpartition} of $T$ which does not cover all $T$. The set difference of $T$ and the (d+1)-square subpartition is called $\color{blue}\blu$. The $8$-square subpartition (for $d=7$) of $T$ is shown in \cref{fig 15}.
\begin{figure}[htbp!]
 \centering
    \includegraphics[width=.3 \textwidth]{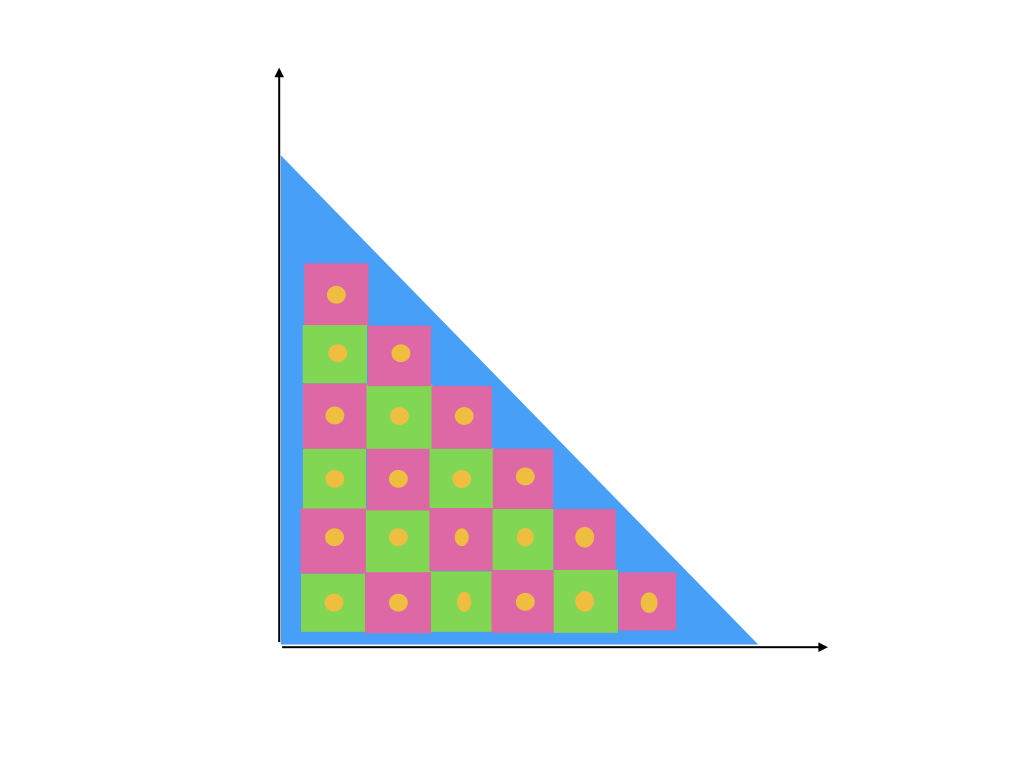}
    \caption{The figure of $8$-square subpartition of T.} \label{fig 15}
  \end{figure}

We define $S_{d+1}$ as follows:
\normalfont 
$$S_{d+1}:=\sum_{0<k<k'<d+1} \frac{4\pi^2}{(d+1)^2}\vol(\frac{2k\pi}{d+1},\frac{2(k'-k)\pi}{d+1}) . $$
\itshape
Where $\frac{4\pi^2}{(d+1)^2}$ is the area of  each square in $(d+1)-$square subpartition.

 We can repeat the same process, by choosing the points  $(\frac{2k\pi}{d+2},\frac{2(k'-k)\pi}{d+2})$, for $0<k<k'<d+2$. Similarly, we have $d+2$-square subpartition of $T$. We consider $S_{d+2}$. As we already mentioned, $S_{d+1}$ and $S_{d+2}$ appear in $m(P_d)$. The difference between the value of the integral and $S_{d}$ for a fixed $d$, is denoted by $E(d)$. For instance, for the $d+1$-square subpartition we have:
\normalfont $$E(d+1)=\left|\iint \limits_{T} \vol(\theta , \alpha) dA - \frac{4\pi^2}{(d+1)^2}\sum_{0<k<k'\leq d} \vol\left(\frac{2k\pi}{d+1},\frac{2(k'-k)\pi}{d+1}\right)\right|,$$
\itshape
 where $\frac{4\pi^2}{(d+1)^2}$ is the area of the squares.
 \end{observation}
  We introduce another notation:
 $$\epsilon(d+1):=\iint \limits_{{\color{blue}\blu}} \vol(\theta , \alpha) dA.$$

 In the following lemma, we compute an upper bound for $\iint \limits_{T} \vol(\theta , \alpha) dA $.
 \begin{lemma}\label{con 9.2}
 We have  $E(d+1) \leq \epsilon(d+1)$. Moreover,
\normalfont
 \begin{align}
     \iint \limits_{T} \vol(\theta , \alpha) dA \leq \epsilon(d+1)+\frac{4\pi^2}{(d+1)^2} \sum_{0<k<k'\leq d}  \vol\left(\frac{2k\pi}{d+1},\frac{2(k'-k)\pi}{d+1}\right).
     \label{eq3}
 \end{align}
 \end{lemma}
 \begin{proof}
According to \cref{ob 9.1}, for a fixed $d$, $T$ is partitioned into $\frac{(d-1)(d-2)}{2}$ squares and the blue part.
 The function $\vol$ is concave and differentiable inside $T$, especially on each square.  Let us focus on arbitrary and fixed square and denote its central point by $(\theta^*,\alpha^*)$. The tangent plane to the graph of $\vol$ at $(\theta^*,\alpha^*)$ denoted by ${\tang_{\vol}(\theta^*,\alpha^*)}$, is located above the graph for all $(\theta,\alpha)$ in the square, so we have:
 
 \begin{align}
\vol(\theta,\alpha) \leq
\tang_{\vol}(\theta^*,\alpha^*).
 \end{align}
The above inequality leads to an upper bound for the double integrals over the square. The volume of the rectangular cuboid with the square as its base and bounded above by the tangent plan of $\vol(\theta,\alpha)$, at $(\theta^*,\alpha^*)$, is greater than $\iint \limits_{\square} \vol(\theta,\alpha) dA$. Hence, we have:
 \begin{align*}
     \iint \limits_{\square}\vol(\theta , \alpha) dA \leq  \iint \limits_{\square}\tang_{\vol}(\theta^* , \alpha^*) dA = \frac{4\pi^2}{(d+1)^2} \vol(\theta^*,\alpha^*).
     \end{align*}
   Therefore, we have:
      \begin{align*}
     \sum_{all\ squares\ inside\ T}\iint \limits_{\square}\vol(\theta , \alpha) dA \leq  \sum_{0<k<k'\leq d} \frac{4\pi^2}{(d+1)^2} \vol(\frac{2k\pi}{d+1},\frac{2(k'-k)\pi}{d+1}),
     \end{align*}
     which is equivalent to the following:
    \begin{align*}
    &\iint \limits_{T} \vol(\theta , \alpha) dA - \sum_{all\ squares\ inside\ T}\iint \limits_{\square}\vol(\theta , \alpha) dA \geq \\
    &\iint \limits_{T} \vol(\theta , \alpha) dA -\sum_{0<k<k'\leq d} \frac{4\pi^2}{(d+1)^2} \vol(\frac{2k\pi}{d+1},\frac{2(k'-k)\pi}{d+1}).
     \end{align*}
    
     Thus, $E(d+1)\leq \epsilon(d+1)$; moreover, we have:
     \begin{align}
     \iint \limits_{T} \vol(\theta , \alpha) dA \leq \epsilon(d+1)+\frac{4\pi^2}{(d+1)^2} \sum_{0<k<k'\leq d}  \vol(\frac{2k\pi}{d+1},\frac{2(k'-k)\pi}{d+1}).
 \end{align}
      \end{proof}
      
 \subsection{A lower bound for the integral}\label{53}
In this section, we define a partition of $T$, which leads to a lower bound for the integral.
\begin{observation}\label{ob 9.2}
\textbf{Triangular partition}:\\
The triangle $T$ is partitioned into the smaller triangles belong to $T_1 \cup T_2$, where $T_1$ and $T_2$ define as follows:
$$T_1:=\bigcup\limits_{i=0}^{d+1}\ \bigcup\limits_{j=0}^{d+1-i}\left\{\left[\left(i\frac{2\pi}{d+1},j\frac{2\pi}{d+1}\right),\left(i\frac{2\pi}{d+1},(j+1)\frac{2\pi}{d+1}\right),\left((i+1)\frac{2\pi}{d+1},j\frac{2\pi}{d+1}\right)\right]\right\},$$
$$T_2:=\bigcup\limits_{i=1}^{d}\ \bigcup\limits_{j=1}^{d+1-i} \left\{\left[\left((i-1)\frac{2\pi}{d+1},j\frac{2\pi}{d+1}\right),\left(i\frac{2\pi}{d+1},j\frac{2\pi}{d+1}\right),\left(i\frac{2\pi}{d+1},(j-1)\frac{2\pi}{d+1}\right)\right]\right\}.$$
In the definition of $T_1$ and $T_2$, $[(i_1,j_1),(i_2,j_2),(i_3,j_3)]$ denotes the triangle with vertices $(i_1,j_1),(i_2,j_2)$, and $(i_3,j_3)$. 
The figure for the $2$-triangular partition is shown in \cref{fig 19}; indeed, the pink and green triangles respectively belong to $T_1$ and $T_2$.
\begin{figure}[htbp!]
  \centering
    \includegraphics[width=.16\textwidth]{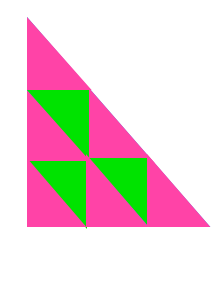}
    \caption{The figure of $2$-triangular partitions of T.} \label{fig 19}
  \end{figure}
   \end{observation}
 \begin{definition}
 The vertices of small triangles, defined in \cref{ob 9.2}, not located on the boundary of $T$ are called \textbf{inner} vertices. The set of all these inner vertices is denoted by $In(T)$.
 \end{definition}
 The following fact leads to an important correspondence between the triangular partition, and the square subpartition. The proof is elementary. 
   \begin{fact}\label{re 5.2}
Each inner vertex of a small triangle, in the $d$-triangular partition, is a central point of a unique square in the $d$-square subpartition.
   \end{fact}
If we restrict $\vol$ to the triangle 
$[a,b,c]$ , since it is concave there exists a unique affine function called $\chi$, such that $\vol(a)=\chi(a), \vol(b)=\chi(b)$,  $\vol(c)=\chi(c)$ and for any $(\theta,\alpha) $ in the triangle we have $\chi(\theta,\alpha) \leq \vol(\theta,\alpha)$. 
Therefore, we have:
 $$\iint \limits_{[a,b,c]} \chi(\theta , \alpha) dA \leq \iint \limits_{[a,b,c]} \vol(\theta , \alpha) dA .$$
\begin{lemma}
Let an arbitrary triangle in $T_1$, introduced in \cref{ob 9.2}, denotes by $[a,b,c]$, and we have:
 \begin{center}
     \begin{tikzpicture}
\draw (0,0) node[anchor=north]{$a$}
  -- (0,1) node[anchor=south]{$b$}
  -- (1,0) node[anchor=north]{$c$}
  -- cycle;
  \path[fill=purple,opacity=0.2] (0,0) -- (1,0) -- (0,1);
\end{tikzpicture}
 \end{center}
\normalfont
 \begin{equation}\label{@}
     \iint \limits_{[a,b,c]}\chi(\theta,\alpha) dA= \area[a,b,c]\left(\frac{1}{3}\vol(a)+ \frac{1}{3}\vol(b)+\frac{1}{3}\vol(c)\right).
 \end{equation} 
 \itshape
Similarly, for another triangle $[b,c,d ]$ belongs to $T_2$, we have:
    \begin{center}
     \begin{tikzpicture}
\draw (1,1) node[anchor=south]{$d$}
  -- (0,1) node[anchor=south]{$b$}
  -- (1,0) node[anchor=north]{$c$}
  -- cycle;
  \path[fill=green,opacity=0.2] (0,1) -- (1,0) -- (1,1);
\end{tikzpicture}
 \end{center}
 \normalfont
 \begin{align}
        \ \iint \limits_{[b,c,d ]}\chi(\theta,\alpha) dA=\area[b,c,d ]\left(\frac{1}{3}\vol(d)+\frac{1}{3}\vol(b)+\frac{1}{3}\vol(c)\right).
      \label{eq2}
      \end{align}
      \end{lemma}
      \begin{proof}
      For an affine function $\chi$, and the triangle  $[a,b,c]$ we have;
$$\iint_{[a,b,c]}\chi(\theta,\alpha) \ dA=\frac{1}{3}\area[a,b,c] \left(\chi(a)+\chi(b)+\chi(c)\right), $$
and after all the computations we have \cref{@}, similarly \cref{eq2} can be proved. 
      \end{proof}
Finally, we are able to compute the lower bound.
 \begin{lemma}\label{lem 9.4}
 We have the following lower bound: 
\normalfont
 \begin{align*}
          \frac{4\pi^2}{(d+1)^2}\sum_{0 < k<k'<d}\vol\left(\frac{2k\pi}{d+1},\frac{2(k'-k)\pi}{d+1}\right)\leq \iint  \limits_{T}\vol(\theta,\alpha)dA.
          \end{align*} 
  \end{lemma}
  \begin{proof}
 We know that:
 $$ \sum_{[b,c,d ] \in T_2}\  \iint \limits_{[b,c,d ]}\vol(\theta,\alpha)dA
  +\sum_{[a,b,c]\in T_1}\ \iint  \limits_{[a,b,c]}\vol(\theta,\alpha) dA = \iint \limits_{T} \vol(\theta , \alpha) dA.$$
By using \ref{@} and  \ref{eq2} in the last equality we have:
       \begin{align*}
            &\sum_{[b,c,d ] \in T_2} \area[b,c,d ]\left(\frac{1}{3}\vol(d)+\frac{1}{3}\vol(b)+\frac{1}{3}\vol(c)\right)+\\&\sum_{[a,b,c]\in T_1}\area[a,b,c]\left(\frac{1}{3}\vol(a)+ \frac{1}{3}\vol(b)+\frac{1}{3}\vol(c)\right)
       \leq \iint \limits_{T} \vol(\theta , \alpha) dA.
      \end{align*}
     In the above computations, the first summation is over the triangles belong in $T_2$ represented by $[b,c,d]$ and the second summation is over the pink triangles belong in $T_1$ represented by $[a,b,c]$ and they have all the same areas, $\frac{2\pi}{(d+1)^2}$, so we can factor it. Notice that for each vertex $a$ the number of times that $\vol(a)$ appears in the summation depends on its location. As we already mentioned, $\vol$ is zero on the boundary of $T$, so let $a$ be an inner vertex of a small triangle. Thus, it appears in exactly $6$ triangles, marked in blue (see \cref{fig 20}). 
     
       \begin{figure}[t]
        \centering
    \includegraphics[width=.25\textwidth]{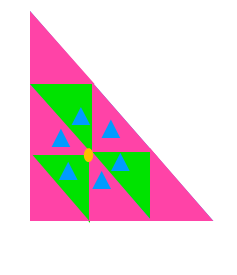}
    \caption{\begin{small}The yellow inner vertex is shared between the six triangles marked in blue.\end{small} } \label{fig 20}
  \end{figure}
     
     Hence, we have:
       
          \begin{align*}
            &\sum_{ [b,c,d ]\  \in  T_2} \area[b,c,d ]\left(\frac{1}{3}\vol(d)+\frac{1}{3}\vol(b)+\frac{1}{3}\vol(c)\right)+\sum_{[a,b,c]\ \in T_1}\area[a,b,c]\left(\frac{1}{3}\vol(a)+ \frac{1}{3}\vol(b)+\frac{1}{3}\vol(c)\right)\\
            &= \frac{4\pi^2}{(d+1)^2}\sum_{a\in In(T) }\frac{6}{3}\vol(a)
           =\frac{4\pi^2}{(d+1)^2}\sum_{0 < k<k'<d}\vol\left(\frac{2k\pi}{d+1},\frac{2(k'-k)\pi}{d+1}\right).
      \end{align*}
    In the last equality we used \cref{re 5.2}, that any inner vertex corresponds to a central point. Finally, we have the lower bound:
     \begin{align}\label{eq4}
         \frac{4\pi^2}{(d+1)^2}\sum_{0 < k<k'<d}\vol\left(\frac{2k\pi}{d+1},\frac{2(k'-k)\pi}{d+1}\right)\leq \iint  \limits_{T}\vol(\theta,\alpha)dA.
     \end{align}
      \end{proof}
      
       \subsection{Finding the limit of \texorpdfstring{$m(P_d)$}{}}
       In this last section, we find the limit of $(m(P_d))_{d \geq 1}$, which was announced in the introduction. First of all, in the following lemma we study the errors $E(d)$, which is another essential tool to find the limit. Using the triangular partition, and square subpartition we prove that when $d$ goes to infinity, $E(d)$ goes to zero faster than $1/d$. 
       \begin{lemma}\label{lem 9.3}$E(d)=o(\frac{1}{d})$.
 \end{lemma}
 \begin{proof}
  We use the upper and lower bounds \ref{eq3} and \ref{eq4}, found respectively in \cref{con 9.2} and \cref{lem 9.4} and we have:
       \begin{align*}
           \frac{4\pi^2}{(d+1)^2} \sum_{0<k<k'\leq d} \vol\left(\frac{2k\pi}{d+1},\frac{2(k'-k)\pi}{d+1}\right)&\leq
       \iint  \limits_{T}\vol(\theta,\alpha)dA\\
       &\leq  \frac{4\pi^2}{(d+1)^2} \sum_{0<k<k'\leq d} \vol\left(\frac{2k\pi}{d+1},\frac{2(k'-k)\pi}{d+1}\right)+\epsilon(d+1).
       \end{align*} 
      Therefore, we conclude;
       $$0\leq \iint  \limits_{T}\vol(\theta,\alpha)dA-
        \frac{4\pi^2}{(d+1)^2} \sum_{0<k<k'\leq d} \vol\left(\frac{2k\pi}{d+1},\frac{2(k'-k)\pi}{d+1}\right)\leq \epsilon (d+1) \leq \maxx. \area({\color{blue}\blu}),$$
       where $\maxx$ is the maximum of $\vol$ on the ${\color{blue}\blu}$ of the triangle. While $d$ is going to infinity the points inside the blue part are approaching the boundary of $T$, where the values of $\vol$ are zero. Hence, the Maximum of $\vol$ in the blue part goes to zero as well. The area of the blue part is $2\pi^2 \frac{3d+1}{(d+1)^2}$ hence, by the definition of $E(d+1)$ we have:
        $$E(d+1)=\iint  \limits_{T}\vol(\theta,\alpha)dA-
        \frac{4\pi^2}{(d+1)^2} \sum_{0<k<k'\leq d} \vol\left(\frac{2k\pi}{d+1},\frac{2(k'-k)\pi}{d+1}\right)\leq  2\pi^2 \frac{3d+1}{(d+1)^2}\maxx.$$
As we explained $\maxx \xrightarrow{d \rightarrow \infty}  0$, so we have $dE(d+1)\xrightarrow{d \rightarrow \infty}  0$. In other words  $E(d+1)=o(\frac{1}{d})$. 
       \end{proof}
       \begin{theorem} \label{th 9.1}
The $\lim_{d \rightarrow \infty} m(P_d)$ exists and it is:
$$\lim_{d\rightarrow \infty}m(P_d) =\frac{9}{2\pi^2} \zeta (3) \simeq 0.548.$$
\end{theorem}
\begin{proof}
By using \cref{lem 9.1} we have:
\begin{align*}
    2\pi m(P_d)= \frac{-2}{d+2} \sum_{0<k<k'\leq d} \vol\left(\frac{2k\pi}{d+1},\frac{2(k'-k)\pi}{d+1}\right)+\frac{2}{d+1} \sum_{0<k<k'\leq d+1} \vol\left(\frac{2k\pi}{d+2},\frac{2(k'-k)\pi}{d+2}\right).
\end{align*}
In order to find $\lim_{d\rightarrow \infty}m(P_d)$ we compute the limit of the R.H.S. Consider the $(d+1)$-square subpartition, so we have:

\begin{align*}
     \iint \limits_{T} \vol(\theta , \alpha) dA =   \frac{4\pi^2}{(d+1)^2} \sum_{0<k<k'\leq d} \vol\left(\frac{2k\pi}{d+1},\frac{2(k'-k)\pi}{d+1}\right)  +E(d+1). 
\end{align*}
Hence, we have:
\begin{align*}
    \frac{-2}{d+2} \sum_{0<k<k'\leq d} \vol\left(\frac{2k\pi}{d+1},\frac{2(k'-k)\pi}{d+1}\right)
    = \frac{-(d+1)^2}{2\pi^2 (d+2)}\iint \limits_{T} \vol(\theta , \alpha) dA +\frac{(d+1)^2}{2\pi^2 (d+2)} E(d+1).
\end{align*} 
We repeat the same process for the case $d+2$ and we have:
\begin{align*}
    \frac{2}{d+1} \sum_{0<k<k'\leq d+1} \vol\left(\frac{2k\pi}{d+2},\frac{2(k'-k)\pi}{d+2}\right)
    = \frac{(d+2)^2}{2\pi^2 (d+1)}\iint \limits_{T} \vol(\theta , \alpha) dA -\frac{(d+2)^2}{2\pi^2 (d+1)} E(d+2).
\end{align*}
We recompute $m(P_d)$ by using the previous information;
\begin{align*}
    2\pi m(P_d)=&\frac{2}{d+1} \sum_{0<k<k'\leq d+1} \vol\left(\frac{2k\pi}{d+2},\frac{2(k'-k)\pi}k{d+2}\right)-\frac{2}{d+2} \sum_{0<k<k'\leq d} \vol\left(\frac{2k\pi}{d+1},\frac{2(k'-k)\pi}{d+1}\right)\\
    = &\frac{(d+2)^2}{2\pi^2 (d+1)}\iint \limits_{T} \vol(\theta , \alpha) dA
    -\frac{(d+1)^2}{2\pi^2 (d+2)}\iint \limits_{T} \vol(\theta , \alpha) dA\\ &\quad +\frac{(d+1)^2}{2\pi^2 (d+2)} E(d+1)-\frac{(d+2)^2}{2\pi^2 (d+1)} E(d+2)\\
    =&\frac{3d^2+8d+7}{4\pi^3(d^2+3d+2)}\iint \limits_{T} \vol(\theta , \alpha) dA+\frac{(d+1)^2}{2\pi^2 (d+2)} E(d+1)-\frac{(d+2)^2}{2\pi^2 (d+1)} E(d+2).
    \end{align*}
 
We find the limit by using the last equality. According to \cref{lem 9.3}, $E(d)=o(\frac{1}{d})$. Hence, 
$\lim_{d\rightarrow \infty}\frac{(d+1)^2}{2\pi^2 (d+2)} E(d+1)=\lim_{d\rightarrow \infty}\frac{(d+2)^2}{2\pi^2 (d+1)} E(d+2)=0$.
Therefore, based on \cref{lem 9.2} we have:
     $$\lim_{d\rightarrow \infty}m(P_d) = \frac{3}{4\pi^3}\iint \limits_{T} \vol(\theta , \alpha) dA= \frac{9}{2\pi^2}\sum_{n=1}^{\infty} \frac{1}{n^3}=\frac{9}{2\pi^2} \zeta (3).$$
     \end{proof}

\bibliographystyle{alpha}
\bibliography{thebib}

\Addresses

\end{document}